\documentclass[12pt]{amsart}
\usepackage[a4paper, hmargin= 2.5cm, vmargin=2cm]{geometry}

\usepackage[T1]{fontenc}
\usepackage[utf8]{inputenc}

\usepackage{amsmath,amsfonts,amssymb,amsthm, amscd}
\usepackage{graphicx}
\usepackage{mathabx, stmaryrd}
\usepackage{fancyhdr}
\usepackage[all]{xy}
\usepackage{color}
\usepackage{color}
\usepackage[dvipsnames]{xcolor}
\usepackage{mathtools}

\usepackage{pgf,tikz}
\usetikzlibrary{arrows}
\usepackage[colorlinks= True, linkcolor=gray, citecolor=blue]{hyperref}
\usepackage[title]{appendix}

\usepackage{pdfsync}
\usepackage{lmodern}

\newcommand{\dem}{Demu\v{s}kin~}
\newcommand{\sthom}{\underline{\hom}}
\newcommand{\Soc}{\operatorname{Soc}}
\newcommand{\cdp}{\operatorname{cd}_{p}}
\newcommand{\cd}{\operatorname{cd}}
\newcommand{\res}{\operatorname{res}}

 \newcommand{\baru}[1]{\underline{#1}}
 \renewcommand{\baro}[1]{\overline{#1}}
 \newcommand{\ove}[1]{[#1]}

\renewcommand{\Im}{\operatorname{Im}}
\renewcommand{\phi}{\varphi}

\newcommand{\Rad}{\operatorname{Rad}}
\newcommand{\End}{\operatorname{End}}
\newcommand{\No}{\operatorname{N}}

\newcommand{\FF}{\mathcal{F}}
\newcommand{\DD}{\mathcal{D}}

\newcommand{\Z}{\mathbf{Z}}
\newcommand{\Q}{\mathbf{Q}}

\newcommand{\F}{\mathbf{F}}
\newcommand{\K}{\mathbf{K}}
\renewcommand{\L}{\mathbf{L}}
\renewcommand{\k}{\mathbf{k}}

\theoremstyle{plain}
\newtheorem{thm}{Theorem}
\newtheorem{lm}[thm]{Lemma}
\newtheorem{prop}[thm]{Proposition}
\newtheorem*{prop*}{Proposition}
\newtheorem*{cor}{Corollary}
\newtheorem*{thm*}{Theorem}

\theoremstyle{remark}

\newenvironment{rem}{\medskip\noindent {\em Remark.}}{\medskip}
\newenvironment{rems}{\medskip\noindent {\em Remarks.}}{\medskip}

\theoremstyle{definition}
\newtheorem*{df}{Definition}
\newtheorem*{exs}{Examples}
\newtheorem*{ex}{Example}
\newtheorem*{cex}{Counter-example}

\newcommand{\floor}[1]{\lfloor #1\rfloor}

\newcommand{\fonction}[5]{\begin{array}{lrcl}
#1\colon & #2 & \longrightarrow & #3 \\
    & #4 & \longmapsto & #5 \end{array}}

\newcommand{\fonctionv}[4]{\begin{array}{rcl}
 #1 & \longrightarrow & #2 \\
 #3 & \longmapsto & #4 \end{array}}

\title{Galois module structure of $p^{\text{th}}$ power classes of abelian extensions of local fields}

\author{Alexandre Eimer}
\address{
Alexandre Eimer\\
Universit\'{e} de Strasbourg \& CNRS\\
Institut de Recherche Math\'{e}matique Avanc\'{e}e\\
UMR 7501, F-67000 Strasbourg, France}
\email{eimer@math.unistra.fr}

\newcommand{\cache}[1]{}

\let\oldtocsection=\tocsection
\let\oldtocsubsection=\tocsubsection
\let\oldtocsubsubsection=\tocsubsubsection

\renewcommand{\tocsection}[2]{\hspace{0em}\oldtocsection{#1}{#2}}
\renewcommand{\tocsubsection}[2]{\hspace{2em}\oldtocsubsection{#1}{#2}}
\renewcommand{\tocsubsubsection}[2]{\hspace{2em}\oldtocsubsubsection{#1}{#2}}

\usepackage{cancel}

\begin{document}

\maketitle

\tableofcontents

\section*{Introduction}

Suppose~$\mathbf{K}/\mathbf{k}$ is a finite Galois extension, and let~$G=Gal(\mathbf{K}/\mathbf{k})$. The group~$J(\mathbf{K})= \mathbf{K}^\times / \mathbf{K}^{\times p}$ of~$p$-th power residue classes in~$\mathbf{K}$, where~$p$ is a prime number, can be seen as an~$\F_p$-vector space, and indeed as an~$\F_p G$-module with the usual Galois action. Understanding~$J(\mathbf{K})$ as such a module can be very useful : see \cite{BLMS} or \cite{MKEG}, for instance, or consider the role played by the $G$-fixed points in \cite{MS1}.

Little is known in general: in~\cite{MiRe}, {\sc J.\ Min\'a\v{c}} and {\sc J. Swallow} study the case when~$G$ is a cyclic~$p$-group (see the references therein for earlier work on the cyclic case); as far as we know, the literature does not contain significant results for other types of groups.

  In this paper, we study numerous $p$-extensions, gradually strengthening the hypotheses: our last and most specific theorem -the real goal of this article- will describe $J(\K)$ when $\K$ is a $p$-\emph{Kummer exstension} over a \emph{local field} that is, an extension $\K/\k$ for which~$G$ is an elementary abelian~$p$-group, assuming that the local field~$\mathbf{k}$ contains a primitive~$p$-th root of unity. As we shall see later those results are in fact closely related to the structure of the maximal pro-$p$-quotient of the absolute Galois group of $\k$: from now on, $\mathcal{G}_{\k}(p)$ denotes this quotient.
  
  Under these assumptions, our description of~$J(\K)$ is fairly complete. When studying a module for a~$p$-group in characteristic~$p$, the information one can hope for is the behaviour after restriction to an arbitrary $\pi$-point (or equivalently to an arbitrary cyclic shifted subgroup, when $G$ is elementary abelian), and the cohomology groups. We address both.

We start with the computation of the cohomology groups $H^{s}(G,J(\K))$.

\begin{thm*}[\hypertarget{GT2}{A}] Let~$\k$ be a local field containing a primitive~$p$-th root of unity, let~$\K$ be a finite, Galois extension of~$\k$, and let~$G=Gal(\K/\k)$. Assume that $G$ is a $p$-group such that the cohomology ring $H^{\bullet}(G,\F_p)$ is Cohen-Macaulay; put $d_{i}(G)=\dim_{\F_{p}}\hat{H}^{i}(G,\F_p)$. Furthermore suppose that $G$ is none of the generalized quaternion groups. We have to distinguish two cases

\begin{enumerate}
\item if the inflation map $\inf\colon H^2(Gal(\K/\k),\F_p)\longrightarrow H^2(\mathcal{G}_{\k}(p),\F_p)$ is zero, the following isomorphisms hold:
\[\left\lbrace\begin{array}{rclr}
H^s (G,J(\K))&\simeq& \hat{H}^{s+2}(G,J(\K))\oplus H^{s-2}(G,J(\K))&s\geq 1 \\
H^0 (G,J(\K))&\simeq& \F_{p}^{d_{2}(G)+(n-d_{1}(G))}
\end{array}\right.
\, . \]
\item if the inflation map $\inf\colon H^2(Gal(\K/\k),\F_p)\longrightarrow H^2(\mathcal{G}_{\k}(p),\F_p)$ is non-zero then
\[\left\lbrace\begin{array}{rcll}H^{1}(G,J(\mathbf{K}))&\simeq &H^{3}(G,\F_{p})& \\
H^{s}(G , J(\mathbf{K}) ) &\simeq & H^{s+2}(G,\F_{p})\oplus H^{s-2}(G, \F_{p})&s\geq 2 \\
H^0(G,J(\K))&\simeq& \F_{p}^{d_{2}(G)-1+(n-d_{1}(G))}
\, .\end{array}\right. \]
\end{enumerate}
 \end{thm*}

The hypothesis on the cohomology ring and on the inflation may seem a little unexpected here, if not completely out of context; yet, as restrictive as it can appear, a large class of groups, which are relevant in Galois theory, verify this hypothesis. For instance, a theorem of Duflot (\cite{Duflot}) implies, that this hypothesis covers the case when $G$ is a non-cyclic elementary abelian $p$-group $G$. For further investigations concerning the importance of this notion in group cohomology, see \cite{BenAlgCo}. Moreover, the hypothesis on the inflation is in fact a rephrasing of a deeper fact which we should extensively study. 

Now, if $\k$ contains a primitive $p$-th root of unity $\xi_{p}$ and $G$ is an elementary abelian $p$-group, which is equivalent to saying that $\K$ is a $p$-Kummer extension, then we can state our second theorem: 

\begin{thm*}[\hypertarget{GT1}{B}]Let $\k$ be a local field such that $\xi_{p}\in\mathbf{k}$ and let $\mathbf{K}/\mathbf{k}$ be an elementary abelian $p$-extension. If $\mathbf{K}/\mathbf{k}$ is not cyclic, then $J(\mathbf{K})^{*}=\mathbf{K}^{\times}/\mathbf{K}^{\times p}$ is a $Gal(\mathbf{K}/\mathbf{k})$-module of constant Jordan type. Furthermore, its stable Jordan type is $[1]^2$. \end{thm*}

See \cite{BensL} for information about modules of constant Jordan type (we recall the basics below, of course). These are actively studied at the moment, and it is perhaps a surprise to see examples arising from field theory.

Before stating other results, we first illustrate our theorem with a concrete example. Consider $p=2$ and $\k = \Q_2$: it contains clearly the primitive $2^{nd}$ root of unity (most commonly known as $-1$). Here, as is well known, the maximal $p$-elementary abelian extension is $\K = \Q_2[\sqrt{\Q_2}] = \Q_2 [\sqrt{2},\sqrt{-1},\sqrt{5}]$ (information of this sort is not used by our methods, and we mention this for concreteness only). Here $G=Gal(\K/\k)\cong C^{3}_{2}$ and~$\dim_{\F_2} J(\K) = 10$. We will explain in this article how to choose generators $x_1, x_2, x_3$ of $G$ and how to construct a convenient basis of $J(\K)$ in order to describe the action. The module structure of $G$ on $J(\K)$ can be given by three matrices, say $M_1, M_2$ and~$M_3$ where $M_{i}$ is the matrix of $x_{i}-id$ in our basis; those matrices commute and are nilpotent. The theorem then says, in this particular case, that the matrix
\[a M_1+b M_2+c M_3=\begin{pmatrix}
& & & & & & & & & \\
& & & & & & & & & \\
& & & & & & & & & \\
& & & & & & & & & \\
& & & & & & & & & \\
b & c & & a & & & & & & \\
& a & & b & & & & & & \\
c & & b & & a & & & & & \\
& & &c &b & & & & & \\
& &a & &c & & & & & \\
\end{pmatrix} \, , \]   
where $a, b, c$ are in an algebraic closure of $\F_2$ (and the blanks are zero) has constant rank (namely 4), as long as $(a,b,c)\neq (0,0,0)$. Of course this may be verified directly.

Among all Kummer extensions, the maximal Kummer extension has retained a lot of attention and its Galois module structure was already investigated -partially- in \cite{MKEG}, when $p=2$. Here we will extend some of the formulae appearing in this paper to the case when~$p$ is odd, restricting attention to local fields (as opposed to the more general~$C$-fields considered in {\em loc.\ cit.}). We gather a lot of explicit information about~$J(K)$, leading in particular to the following.

\begin{prop*}[\hypertarget{Inv}{C}]Let $\K$ be the maximal $p$-abelian extension of a local field $\k$, which possesses a primitive $p$-th root of unity. Then the $Gal(\K/\k)$-module $J(\K)$ has socle length:
\[l(J(\K))=n (p-1)-1 \, , \]
where $n$ is the dimension of $\k^{\times}/\k^{\times p}$ as an $\F_p$ vector space. Furthermore, if $p$ is odd, the following equality holds
\[\dim_{\F_p}\Soc^{2}(J(\K))/\Soc(J(\K))=\frac{(n-2)(n^2 +5n +3)}{3}\, . \] \end{prop*}

Both theorems will appear as consequences of the following key statement. To formulate it, we must recall that~$\F_p G$-modules are best studied in the {\em stable category}, in which~$\underline{\hom}(M,N)$ consists of all usual modules homomorphisms modulo those morphisms which factor through a projective module. This category is triangulated, and is thus endowed with an invertible endofunctor~$\Omega $, frequently called the {\em Heller shift}. At the heart of this paper is the next result, which we formulate in vague terms for now.

\begin{prop*}[\hypertarget{GT3}{D}]
  
Let~$G= Gal(\K/\k)$ as above, which we assume to be a~$p$-group. For~$s \in \{ -1, 2 \}$, there exists a module~$\omega_s(\F_p)$ which is stably isomorphic to~$\Omega^s(\F_p)$, and an exact sequence
   
\[\xymatrix{0 \ar[r] & \omega_{-1}(\F_p) \ar[r]^{\kappa} & \omega_{2}(\F_p) \ar[r] & J(\K) \ar[r] & 0 } \, . \]

\end{prop*}

Strangely enough this result is valid with no further assumption on~$G$, although our main theorems above require more. Note that the condition on the inflation in theorem~\hyperlink{GT1}{A} is expressed in a more natural way using $\kappa$. 

The map~$\kappa $ seems to be the most profound object entering our discussions. As a stable map $\Omega^{-1}(\F_p) \longrightarrow \Omega^2(\F_p)$, it can be interpreted as a class in $\hat H^{-3}(G, \F_p) = H_2(G, \F_p)$, or as a linear map $H^2(G, \F_p) \longrightarrow \F_p$ ; in the case when~$\K/\k$ is the maximal Kummer extension,  we describe explicitly its values on the usual basis for~$H^2(G, \F_p)$ when~$G$ is elementary abelian. This should not astonish the reader: indeed it should be compared to a theorem of {\sc U. Jannsen} (\cite[Theorem 5.4]{Jannsen}), where {\sc Jannsen} proved that for \emph{any} Galois $p$-extension $\K/\k$ of a local field $\k$, the $Gal(\K/\k)$-module $(\K^{\times}/\K^{\times p})^{*}$ is determined by the group of $p$-th roots of unity in $\K$ and a canonical class $\chi$ in the linear dual of $H^2(Gal(\K/\k),\mu_{\k}(p))$.

In order to establish the existence of the exact sequence, as befits Galois theory, we turn questions of field theory into questions of group theory. A bit more precisely, let~$\mathcal{G} $ be a pro-$p$-group and $\mathcal{H}$ a closed, normal subgroup, let~$\Phi (\mathcal{H} )$ denote its Frattini subgroup, that is

\[ \Phi (\mathcal{H} ) = \mathcal{H}^p(\mathcal{H} , \mathcal{H} ) \, .   \]

Thus~$\mathcal{H} / \Phi (\mathcal{H} )$ is the largest quotient of~$\mathcal{H} $ which is elementary abelian. After a little translation, our results will be about ~$J=\mathcal{H} /\Phi (\mathcal{H} )$, seen as an~$\F_p G$-module, where~$G= \mathcal{G} / \mathcal{H} $. We will investigate this when~$\mathcal{G} $ is a free pro-$p$-group, and then in the case when~$\mathcal{G} $ is the largest~$p$-quotient of~$Gal(\bar{\mathbf{k}}/\mathbf{k})$, which is a Demu\v{s}kin group, provided the fact that $\mathbf{k}$ is a local field with a primitive $p^{th}$ root of unity.

\bigskip

{\em Organization of the paper.} In the first section, we recall some basic facts about modules of constant Jordan type and Kummer theory; this culminates with the statement of two ``Main Theorems'', more general (and more technical) than the above, covering for example the case when~$\k$ does not contain enough roots of unity. Then in section \ref{P2} we construct the short exact sequence making its appearance in \hyperlink{GT3}{D}, and use it to prove the Main Theorems (in fact refining them in certain cases). Section \ref{P3} completes the picture with a focus on the maximal $p$-elementary abelian extension, which means in group theory that $\mathcal{H}=\Phi(\mathcal{G})$: we give some explicit computations which hopefully flesh out the two previous sections and take profit of them to prove the announced formulae in proposition \hyperlink{Inv}{C}. Finally, in section \ref{P4}, we address the case of two peculiar extensions over $\Q_3 (\zeta_3)$ in order to show that both situations, distinguished in theorem \hyperlink{GT1}{A}, occur.

\bigskip

{\em Acknowledgements.} This article is a part of our PhD thesis: it would not have been possible without the help of our advisor Pierre Guillot, and we thank him for his precious piece of advice and the numerous corrections on this paper. Furthermore, we would like to express our gratitude towards Dave Benson, who suggested some arguments which enabled us to considerably improve our results and proofs: we owe him a lot.

\section{Background material}

\subsection{Modules of constant Jordan type}\label{JT}

Here, we quickly recall some basic facts about modules of constant Jordan type. A more general approach is contained in the fundamental article of {\sc Carlson}, {\sc Friedlander} and {\sc Pevtosva} (\cite{JordanType}), treating the case of group schemes and a more exhaustive one about elementary abelian $p$-groups in {\sc Benson}'s book (\cite{BensL}).

In this section, $G$ is a $p$-group and $\F$ an algebraically closed field of characteristic $p$. Modules of constant Jordan type were introduced in order to properly extend what was already known for $\F C_{p}$-modules (here $C_{p}$ stands for the cyclic group of order $p$). Indeed, in this case, a module $M$ is completely described by the action of a given generator $x_1$ of $C_{p}$, more precisely we have the following theorem:

\begin{thm}The $\F C_{p}$-modules of dimension $n$ (over $\F$) are in bijection with the partition of $n$ by parts of size no greater than $p$, up to isomorphism.\end{thm}

\begin{proof}Let $M$ be an $\F C_{p}$-module. By a slight but classical abuse of notation, we will write $x_1$ for a generator of $C_{p}$, the associated element in $\F C_{p}$ and the associated endomorphism in $\End_{\F}(M)$.

Now, remember that, since $\F$ is of characteristic $p$, the beginner's dream is true in $\F C_{p}$:
\[x^{p}_1 - 1 = (x_1 -1)^{p} \, ; \]  
however $x_{1}^{p}$ is zero in $C_{p}$, so that $x_1-\operatorname{Id}$ is a nilpotent endomorphism; it is commonly known that such morphisms are classified by their Jordan type. The Jordan form of $x_1 - \operatorname{Id}$ gives a decomposition of $M$ into cyclic $\F C_{p}$-modules. Remark that the Jordan type of $x_1- Id$ does not depend upon the choice of $x_1$, so that this is well defined, and we can state that two $\F C_{p}$-modules having the same Jordan type are isomorphic. Hence the announced classification.

\end{proof}

\begin{rem}The Jordan type $[1]^{n_1}\ldots [p]^{n_p}$ of the morphism $x_1$ is called the Jordan type of the module $M$.\end{rem}

For such a module $M$, $n_{j}(M)$ denotes the number of its blocks of length $j$, where $j$ verifies the inequalities $1\leq j \leq p$. We can easily relate this block decomposition of a cyclic module to the dimension of the cohomology groups, by the following lemma.

\begin{lm}\label{CoJt}Let $M$ be an $\F E_1$-module, then
\[\left\lbrace\begin{array}{rcl}
\dim H^0 (E_1,M)&=& \sum\limits_{j=1}^{p} n_{j}(M)\, , \\
\dim H^1 (E_1,M)&=& \sum\limits_{j=1}^{p -1}n_{j}(M)\, .
\end{array}\right. \]  \end{lm}

The proof of this innocent lemma is left to the reader; despite its simplicity it will be our key argument in the proof of theorem \hyperlink{GT1}{A},  and in the treatment of one of our examples.

In order to adapt the previous method to non-cyclic groups, we will use the language of $\pi$-points. But first, we have to recall some basic facts: let $A_1$ and $A_2$ be two $\F$-algebras and let $\alpha\colon A_1\longrightarrow A_2$ be a morphism of $\F$-algebras. The morphism $\alpha$ induces a functor from the category of \emph{right} modules of finite type $\mathfrak{mod}(A_1)$ to the category of \emph{right} modules of finite type $\mathfrak{mod}(A_2)$. It is in fact known by all that every $A_2$-module $M$ can be turned into an $A_1$-module using this external law:
\[ x\cdot a_1=x\cdot\alpha(a_1)\, ,\quad \forall x\in M\, , \forall a_1\in A_1 \, ,\] 
and it is easily seen that this construction is functorial. From now on, $\alpha_{*}$ will denote the functor induced by such a morphism $\alpha$. It should be immediately remarked that this functor verifies multiple properties: it is for instance additive and exact: two small facts we shall make good use of. Keep in mind that we made the choice to study right modules and not the usual left modules: the reason for this rather unconventional choice will be put into light in Section \ref{P3}.

\begin{df}A $\pi$-point is a morphism of algebra
\[\alpha\colon\F[T]/(T^{p})\longrightarrow\F G \, , \] which is flat, that is: $\alpha_{*}(\F G)$ is a projective $\F[T]/(T^p)$-module. \end{df}

\begin{rems}Two remarks have to be made.

\paragraph{Module structure}It should be noticed that $\F [T] / (T^{p})$ is isomorphic to $\F E_1$, where $E_1$ is the $p$-elementary abelian group of rank 1. Let us choose a generator $\gamma$ of $E_1$ and set $\Gamma=\gamma-1\in\F E_1$, so that the isomorphism between $\F [T] / (T^p )$ and $\F E_1 $ is simply given by \[f\colon\Gamma\mapsto T \, .\]
Therefore, according to our previous discussion, every $\pi$-point $\alpha$ enables us to give a structure of $\F E_1$-module to every $\F G$-module.

\paragraph{Flatness condition}The flatness condition shall not be neglected: indeed, thanks to it, the $\F E_1$-module $\alpha_{*}(P)$ is projective if $P$ is a projective $\F G$-module. This simple fact has major consequences relative to the cohomology: $\alpha_{*}$ induces a morphism of cohomological functors
\[\res_{\alpha}\colon\hat{H}(G,-)\longrightarrow\hat{H}(E_1,-)\, . \]
Such map is written like a restriction, because it should be thought of as such: indeed, all expected properties of the restriction exposed in any textbook (for instance \cite[Part III]{Gui}) remain true.\end{rems}

\begin{df}Let $\alpha,\beta$ be two $\pi$-points.Then the $\pi$-points $\alpha$ and $\beta$ are said to be equivalent if for every module $M$, $\alpha_{*}(M)$ is projective if and only if $\beta_{*}(M)$ is projective. \end{df}

\begin{ex}
Let us consider more closely the case where $G$ is an elementary abelian $p$-group of rank $r$, which will be written $E_{r}$.

It is well-known that given a basis $x_1,\ldots,x_r$ of $E_{r}$ the elements denoted $X_{i}=x_{i}-1\in \F E_{r}$ form a basis of $\F E_{r}$ as an $\F$-algebra ($i$ is obviously between $1$ and $r$); more precisely, $\F E_{r}\simeq\F [X_1,\ldots,X_{r}]/(X_{i}^{p})$, where $(X_{i}^{p})$ denotes the ideal generated by the various $X_{i}^p$. It could be shown that in this case every $\pi$-point is of the form

\[\fonction{\alpha}{\F[T]/(T^{p})}{\F G}{[T]}{\gamma\in \Rad (\F E_{r})} \, . \]

Now, according to \cite[lemma 6.4]{CarlVCRM}, two $\pi$-points $\alpha,\beta$ on $\F E_r$ are equivalent if and only if the image of $\alpha-\beta$ lies in $I^2 (E_{r})$ -where $I(E_n)$ is the augmentation ideal. Thus a $\pi$-point -up to equivalence- is simply a morphism
\[\fonction{\alpha}{\F[T]/(T^{p})}{\F G}{[T]}{\sum\limits_{i=1}^{r}a_{i}X_{i}} \, . \]

\end{ex}

Now, it is time to introduce the modules of constant Jordan type.

\begin{df}An $\F G$-module $M$ is said to be of constant Jordan type $[a_1]^{m_1}\ldots [a_{l}]^{m_l}$, if the Jordan type of $\alpha_{*}(M)$ is $[a_1]^{m_1}\ldots [a_{l}]^{m_l}$ for every $\pi$-point $\alpha$. Its Jordan type is called the Jordan type of $M$. If we omit the block of length $p$, we speak of \emph{stable} Jordan type. \end{df}

\begin{rem}
If $M$ is an $\F' G$-module where $\F'$ is a field of characteristic $p$ which is \emph{not} algebraically closed, we say that $M$ is of constant Jordan type if $\bar{\F '}\otimes_{\F'} M$ is of constant Jordan type, where $\bar{\F'}$ is of course an algebraic closure of $\F'$.

\end{rem}

\begin{exs}\begin{enumerate}
\item
  Projective modules are a first example; indeed, $\F G$-projective modules are just direct sums of copies of $\F G$. Since a $\pi$-point $\alpha$ is flat, if $P$ is a projective $\F G$-module, then $\alpha_{*}(P)$ is a projective $\F E_1$-module: thus, it is a direct sum of copies of $\F E_1$ or, in other words, of blocks of size $p$. Hence a projective module is of constant Jordan type $[p]^{\frac{\dim P}{p}}$. Note in this case if $G$ is an elementary abelian group, the converse is true: it is the famous {\sc Dade}'s lemma. (see \cite[Lemma 1.9.5]{BensL})
  
\item \label{Ex2} Here is an example that foreshadows the proof of theorem \hyperlink{GT2}{B}. Let $I(G)^*$ be the dual of the augmentation ideal of $G$, and let us show directly that it is a module of constant Jordan type. Indeed, this module is defined by the short exact sequence:
\[\xymatrix{0 \ar[r] & \F \ar[r] & \F G \ar[r]^{\varepsilon^*} & I(G)^* \ar[r] & 0} \, , \] 
where $\varepsilon$ is the norm map. Let $\alpha$-be a $\pi$-point. Since $\alpha_*$ is an exact functor and sends projective modules to projective modules, by taking a look at the long exact sequence in cohomology, we have that
\[\dim H^1(E_1,\alpha_{*}(I(G)^*))=1 \, . \]
Using exactly the same trick, but one degree lower, and the well-known dimension of $I(G)^*$, we can compute the following dimension:
\[\dim H^0 (E_1,\alpha_{*}(I(G)^*))=\frac{|G|}{p} \, . \]
Therefore, according to the lemma \ref{CoJt}, we get \[n_{p}(\alpha_{*}(M))=\dim H^0 (E_1,\alpha_* (I(G)^*))-\dim H^1(E_1,\alpha_* (I(G)^*))=\frac{|G|}{p}-1\, .\] Furthermore, there is only one block whose size is not $p$. Because we have the following equality
\[\dim \alpha_{*}(I(G)^*)=\dim I(G)^* = |G|-1\, ,\] we deduce that the size of this block is $p-1$. Thus the module is of constant Jordan type $[p-1][p]^{\frac{|G|}{p}-1}$. We will later state a proposition (\ref{HelCst}), which will make obvious that $I(G)^*$ is of constant Jordan type.

 \end{enumerate}
 \end{exs}

\begin{cex} Consider the $\F E_3$-module given by generators and relations:

\[M=\langle r|r\cdot X_{1}^{p-1}X_{2}^{p-1}=0 \rangle\, ; \]
this module is not of constant Jordan type. Indeed, consider the following two $\pi$-points:
\[\alpha\colon\Gamma\mapsto X_1\, ,\quad\beta\colon\Gamma\mapsto X_3 \, . \]
It is easy to see that $\alpha_{*}(M)$ has block decomposition $[p-1]^{p}[p]^{p(p-1)}$, each block of size $p$ being generated by $r\cdot X_{2}^{j}X_{3}^{k}$ (with $j\neq p-1$ and $0\leq k \leq p-1$) and the blocks of size $p-1$ by the $r\cdot X_{2}^{p-1}X_{3}^{k}$, whereas $\beta_{*}(M)$ has block decomposition $[p]^{p^2 -1}$, all those blocks being generated by the $r\cdot X_{1}^{j}X_{2}^{k}$ where $i,j\in\{0,\ldots,p-1\}$ and $(j,k)$ is different from $(p-1,p-1)$.
\end{cex}

Further elementary examples of modules of constant Jordan type will be given later. The following proposition from \cite{JordanType} enables us to build more modules of constant Jordan type.

\begin{prop}\label{OpJt}The full subcategory $\mathfrak{ctJt}(\F G)$ of $\mathfrak{mod}(\F G)$ whose objects are modules of constant Jordan type is closed under direct sums, tensor products and taking the linear dual. \end{prop}

Remember that if $M$ is an $\F G$-module, the linear dual $M^*$ of $M$ is the module which, as an $\F$-vector space, is isomorphic to $\hom_{\F}(M,\F_p)$ and whose module structure is simply given by the law
\[f\cdot g\colon x\mapsto f(x\cdot g^{-1})\, ,\quad\forall f\in M^*\, , \forall g\in G \, .\]
Within only a few lines, it can be proved that the direct sum of two modules of constant Jordan type is of constant Jordan type. It is less obvious that this class is stable under tensor product and taking the linear dual (for a proof, see \cite[Prop. 1.8, cor. 4.3]{JordanType}). The latest can even reserve some surprises: $M$ is of constant Jordan type if and only if $M^*$ is so, yet the Jordan type of $M$ is not necessarily the same as the one of $M^*$! For a counter-example, see for instance \cite[Example 1.13.1]{BensL}.

\subsection{Local fields}
\label{Local fields}
Now, it is time to introduce the extensions which are at the heart of this article. Let $\mathbf{k}$ be a local field which contains a primitive $p$-th root of unity and let us fix $\bar{\mathbf{k}}$ an algebraic closure of $\mathbf{k}$. We set $\K$ a Galois $p$-extension of finite type of $\k$ whose Galois group called $G$ verifies that its cohomology ring (with coefficients in $\F_p$) is Cohen-Macaulay. As recalled in the introduction, because of Duflot's theorem, if $G$ is an elementary abelian $p$-group which is not cyclic, this assumption holds.

In spite of their ingenuity, such groups play a major role in field theory.

The study of $J(\K) = \K^\times/\K^{\times p}$ will lead us to the study of some other extensions, as we shall see hereafter. Let us put $\mathcal{R} =\{ \alpha\in\bar{\mathbf{k}}|\alpha^{p}\in \K\}$ and $\mathbf{K}^{(2)}=\K[\mathcal{R}]$, so that we have the following diagram of extensions.

\[\xymatrix{\mathbf{K}^{(2)} \ar@{-}[d] \\ \K \ar@{-}[d] \\ \mathbf{k}} \]

According to the previous diagram we have the following exact sequence 
\[\xymatrix{0 \ar[r] & Gal(\mathbf{K}^{(2)}/\K) \ar[r] & Gal(\mathbf{K}^{(2)}/\mathbf{k}) \ar[r] & Gal(\K/\mathbf{k}) \ar[r] & 0}\, . \]

Therefore it is abundantly clear that $G:= Gal(\K/\k)$ acts on $Gal(\mathbf{K}^{(2)}/\K)$ by conjugation.  Furthermore, as an~$\F_p G$-module $Gal(\K ^{(2)}/\K)$ is related to~$J(\mathbf{K})$ via {\em Kummer theory.} Let us recall the basics.

Remember first that a field extension $\L/\L_0$ is an $n$-\emph{Kummer extension} if it is simply a Galois extension such that $Gal(\L/\L_0)$ is an abelian group of exponent dividing $n$, and if~$\L_0$ contains a primitive $n$-th root of unity. Like in Galois theory, there is a correspondence theorem. 

\label{KT}
\begin{thm}[Kummer theory]Let $\mathbf{L}_0$ be a field containing a primitive $n$-th root of unity and fix an algebraic closure $\bar{\mathbf{L}}_0$ of $\mathbf{L}_0$. The $n$-Kummer extensions $\mathbf{L}$ of $\mathbf{L}_0$ contained in $\bar{\mathbf{L}}_0$ are in 1-to-1 correspondence with the subgroups of $\mathbf{L}_{0}^{\times}/\mathbf{L}_{0}^{\times n}$; moreover the correspondence maps the subgroup~$H$ to the field $\mathbf{L}_{0}[x^{\frac{1}{n}},[x]\in H]$, and is thus order-preserving. Finally, if a field $\mathbf{K}$ and a subgroup $H$ are in correspondence, then $\hom(Gal(\mathbf{L}/\mathbf{L}_{0}),\Z/n\Z)\cong H$.

\end{thm}

See~\cite[Theorem 1.25]{Gui}. From this theorem and the previous considerations we can deduce two key facts (with~$n=p$ in both cases).

First, the group $Gal(\mathbf{K}^{(2)}/\K)$ is simply an elementary abelian $p$-groups and in particular an $\F_{p}$-vector space. 

Secondly, the theorem applied to the base field~$\K$ implies, by maximality, that~$\K ^{(2)}$ is in correspondence with~$J(\mathbf{K})$ (note that all~$p$-Kummer extensions of~$\K$ are contained in~$\K ^{(2)}$). It follows that $\hom(Gal(\mathbf{K}^{(2)}/\K), \F_p)$ is isomorphic to $J(\mathbf{K})$, and this is really an isomorphism of $\F_p G$ \emph{modules}: indeed this is the refinement brought by equivariant Kummer theory (see \cite[Theorem 1.26]{Gui}).

A more elaborate result, which we call Tate duality (cf. \cite[Theorem 13.21]{Gui}), states that $J(\mathbf{K})$ is self-dual, as a module, as long as $\mathbf{k}$ contains a primitive $p$-th root of unity, which is fortunately the case here.

We summarize this discussion in the following lemma:

\begin{lm} There is an isomorphism of $\F_pG$-modules between $Gal(\mathbf{K}^{(2)}/\K)$, $J(\mathbf{K})$ and $J(\mathbf{K})^*$.\end{lm}

From now on, we set $J= J(\mathbf{K})$. According to the previous lemma and the proposition \ref{OpJt}, instead of trying to study $J$, we can turn our attention to its dual namely~$Gal(\K ^{(2)} / \K)$ and use techniques from group theory.

Let us write~$\L(p)$ for the largest pro-$p$ extension of the field~$\L$ (contained in a fixed algebraic closure), and put~$\mathcal{G}_\L(p) = Gal(\L(p)/\L)$. Let $\mathcal{H}$ the subgroup of $\mathcal{G}_{\k(p)}$ in Galois correspondence with $\K$. It is not hard to see that $\mathbf{K}^{(2)}$ is in correspondence with $\Phi(\mathcal{H})$, using the maximality condition defining these extension and the one defining the Frattini subgroup. We can therefore state the following lemma:

\begin{lm} \label{lem-J-iso-frattini} There is an isomorphism $J^* \cong \mathcal{H}/\Phi (\mathcal{H})$, as modules over~$\F_p G$ where $G = \mathcal{G}_{\mathbf{k}}(p)/\mathcal{H}$.\label{baba} \end{lm}

\subsection{The main theorems} \label{subsec-main}

Thanks to the previous lemma, we have completely translated the problem arising from Galois theory into a group-theoretic one; not only does this formulation enable us to solve the problem, but we can now rephrase all major results of this article in the two following theorems. 

\begin{thm}[Main theorem I]\label{MT1} Let $\mathbf{k}$ be a local field and let $\mathcal{G}_{\mathbf{k}}(p)$ be the Galois group of a maximal pro-$p$-extension. Consider a closed normal subgroup of finite index $\mathcal{H}$ of $\mathcal{G}_{\k}(p)$ and put $J = \hom(\mathcal{H}/\Phi(\mathcal{H}),\F_p)$ and $G=\mathcal{G}_{\mathbf{k}}(p)/\mathcal{H}$. Remember that $d_{i}(G)=\dim_{\F_p}\hat{H}^{i}(G,\F_p)$.

We have the following possibilities for the cohomology of~$G$ with coefficients in $J$.

\begin{enumerate}
\item If $\mathbf{k}$ does not contain a primitive $p^{th}$ root of unity, then for all~$s \in \Z$:
\[\left\lbrace\begin{array}{rcl}
\hat{H}^{s}(G,J)&=&\hat{H}^{s+2}(G, \F_{p}) \\ H^0(G,J)&\simeq&\F_{p}^{d_{2}(G)+n-d_{1}(G)} \end{array}\right. \, . \] 
\item Suppose that $\xi_{p}\in\mathbf{k}$ and $H^{\bullet}(G,\F_p)$ is a Cohen-Macaulay ring, then we have to distinguish between two cases

\begin{enumerate}
\item if the inflation map $\inf\colon H^2(Gal(\K/\k),\F_p)\longrightarrow H^2(\mathcal{G}_{\k}(p),\F_p)$ is zero, the following isomorphisms hold:
\[\left\lbrace\begin{array}{rclr}
H^s (G,J(\K))&\simeq& \hat{H}^{s+2}(G,J(\K))\oplus H^{s-2}(G,J(\K))\, ,&s\geq 1 \\
H^0 (G,J(\K))&\simeq& \F_{p}^{d_{2}(G)+(n-d_{1}(G))}
\end{array}\right.
\, . \]
\item if the inflation map $\inf\colon H^2(Gal(\K/\k),\F_p)\longrightarrow H^2(\mathcal{G}_{\k}(p),\F_p)$ is non-zero then
\[\left\lbrace\begin{array}{rcll}H^{1}(G,J(\mathbf{K}))&\simeq &H^{3}(G,\F_{p})& \\
H^{s}(G , J(\mathbf{K}) ) &\simeq & H^{s+2}(G,\F_{p})\oplus H^{s-2}(G, \F_{p})\, ,&s\geq 2 \\
H^0(G,J(\K))&\simeq& \F_{p}^{d_{2}(G)-1+(n-d_{1}(G))}
\, .\end{array}\right. \]
\end{enumerate}
\end{enumerate}
\end{thm}

\begin{thm}[Main theorem II]\label{MT2}Under the same hypothesis over $\k$ and using the same notation, the $G$-module $J^*$ 
\begin{itemize}
\item is of constant Jordan type, and its stable Jordan type is $[1]$, if $\mathbf{k}$ does not contain a primitive $p^{th}$-root of unity,
\item is of constant Jordan type, and its stable Jordan type is $[1]^{2}$ , if $\mathbf{k}$ does contain a primitive $p^{th}$-root of unity and $G$ is an elementary abelian $p$-group.
\end{itemize}

\end{thm}

Note that in the case where $\mathbf{k}$ does not contain a primitive $p^{th}$-root of unity, there is no such thing as Kummer theory; therefore there is no isomorphism between $\K^{\times}/\K^{\times p}$ and $\mathcal{H}/\Phi(\mathcal{H})$, so that this formulation of the theorem is the only one available. It should be remarked that $J$ reflects some differences between those fields and their arithmetic.

The next section of the paper is devoted to the proof of those theorems (which implies, in particular, the statements of Theorems \hyperlink{GT1}{A} and \hyperlink{GT2}{B} from the introduction, of course). Before we turn to this however, we need to continue with more background material.

\subsection{\dem groups} \label{subsec-demushkin}

The Galois groups of maximal $p$-extensions of local fields are explicitly known: indeed if $\mathbf{L}$ is a local field such that $\xi_{p}\in\mathbf{L}$, then $\mathcal{G}_{\mathbf{L}}(p)$ is a \emph{Demu\v{s}kin group}. A presentation by generators and relations of such groups was given by {\sc J. Labute} (see \cite{Lab}), which we recall first for $p\neq 2$:
\[\mathcal{D}_{k,2s}=\langle x_{1},\ldots,x_{2s}|x_{1}^{p^{k}}(x_1,x_2)(x_3,x_4)\ldots (x_{2s-1},x_{2s})=1 \rangle \, , \]
where $k$ is the maximal integer such that $\xi_{p^{k}}\in\mathbf{L}$ and $2s$ is the dimension of $J(\mathbf{L})$.

When $p=2$, the relation in the Demu\v{s}kin group changes. If the number of generators is odd, it becomes
\[\mathcal{D}_{f,n=2s+1}=\langle x_1,\ldots,x_{2s+1}|x_{1}^{2}x_{2}^{f}(x_2,x_3)(x_4,x_5)\ldots (x_{2s},x_{2s+1})=1 \rangle \, . \]
However, if the number of generators is even, the relation is either
\[\mathcal{D}_{f,n=2s}=\langle x_1,\ldots,x_{2s}|x_{1}^{2+2^{f}}(x_1,x_2)(x_3,x_4)\ldots (x_{2s-1},x_{2s})=1\rangle \, , \]
or
\[\mathcal{D}'_{f,n=2s}=\langle x_1,\ldots,x_{2s}|x_{1}^{2}(x_1,x_2)x_{3}^{2^{f}}(x_3,x_4)\ldots (x_{2s-1},x_{2s})=1\rangle \, . \]

In each case $f$ is an integer such that $f\geq 2$.

We complete this review of the possible descriptions for~$\mathcal{G}_\L(p)$ with the case when~$\L$ does not contain a primitive~$p^{th}$-root of unity: in this situation $\mathcal{G}_\L(p)$ is just a free prop-$p$-group (\cite[Theorem 3, II, \S5]{CoGal}).

\subsection{The stable module category and Heller shifts} \label{subsec-heller-basics}

We have to introduce some new modules: our key argument is yet very simple (it is just a short exact sequence), but we have to explain some classical notation and objects. Here we just follow \cite[\S 2.5 sq.]{CoRi}, so we consider a finite group $G$ and a field~$\F$ (whose characteristic~$p$ typically divides the order of~$G$).

Let $M$ be an $\F G$-module, let $\pi:P\longrightarrow M$ an epimorphism from a projective module onto $M$. Its kernel denoted $\Omega(M)$ is called the \emph{Heller shift} of $M$; it always exists, however it is only defined up to a projective summand. That is why we have to introduce the \emph{stable} module category $\underline{\mathfrak{mod}}(\F G)$ whose objects are but $\F G$-modules, and whose hom sets, written $\sthom$, are defined by 
\[ \underline{\hom}(M,N)=\hom_{\F G}(M,N)/P_{M,N}\, , \]
where $P_{M,N}$ is the subspace of morphisms which factor through a projective. Then $\Omega$ becomes a well-defined functor on the stable category.

We should immediately remark that we can iterate $\Omega$ and set without ambiguity $\Omega^{i+1} (M)=\Omega(\Omega^{i}(M))$ and so on. Dualizing this construction (i.e. taking the cokernel of a monomorphism from $M$ into a projective module) gives birth to $\Omega^{-1}(M)$ and then we can again iterate such a construction. We would like to emphasize the fact that $\Omega(M)$ is not well-defined in the category of modules but in the stable category; usually  $\omega(M)$ will be our notation for \emph{some} module whose image in $\underline{\mathfrak{mod}}(\F G)$ is isomorphic to $\Omega(M)$, though we will repeat this for emphasis.

Furthermore the previous construction is natural: $\Omega$ is an endo-functor of $\underline{\mathfrak{mod}}(\F G)$ and an equivalence of category whose quasi inverse is, as expected, $\Omega^{-1}$. We may hope that it adds a little bit of structure to $\underline{\mathfrak{mod}}(\F G)$. In fact $(\underline{\mathfrak{mod}}(\F G),\Omega^{-1})$ is a triangulated category: given a short exact sequence in $\mathfrak{mod}(\F G)$
\[\xymatrix{0 \ar[r] & L \ar[r]^{\alpha} & M \ar[r]^{\beta} & N \ar[r] & 0}\, , \]
it is possible to build maps $(\tilde{\alpha},\tilde{\beta},\gamma)$ in the stable module category in such a way the class of $\alpha$ (resp.$\beta$) is $\tilde{\alpha}$ (resp. $\tilde{\beta}$) and $\gamma\colon N\longrightarrow\Omega^{-1}(L)$, so that the distinguished triangles are all the triangles isomorphic to one of the form:

\begin{equation}\xymatrix{L \ar[r]^{\tilde{\alpha}} & M \ar[r]^{\tilde{\beta}} & N \ar[r]^{\gamma} & \Omega^{-1}(L) }\, . \label{Trg} 
\end{equation} We summarize in the following proposition :

\begin{prop}The additive category $\underline{\mathfrak{mod}}(\F G)$, equipped with the functor $\Omega^{-1}$ and whose distinguished triangles are the ones described above, is a triangulated category. \end{prop} 

\begin{rems}\begin{enumerate}
\item
One of the main interests of $\Omega$ is the fact that it may give a new definition of Tate cohomology, namely 
\[\hat{H}^{k}(G,M)\simeq\sthom(\Omega^{s+k}(\F_p),\Omega^{s}(M))\, , \quad\forall s,\,  k\in\Z\, . \] \label{Stable}
\item Let us consider the following exact sequence in $\mathfrak{mod}(\F G)$ :
\[ \xymatrix{0 \ar[r] & L \ar[r]^{\alpha} & M \ar[r]^{\beta} & N \ar[r] & 0}\, .\]
Since the cone of a map in a triangulated category is unique up to isomorphism, the module $N$ is stably isomorphic to the cone of $\alpha$.
\end{enumerate}
\end{rems}

We will consider in particular the stable modules $\Omega^{-1}(\F)$ and $\Omega^2( \F)$ when~$G$ is a $p$-group, indeed, when $\mathcal{G}_{\k}(p)$ is a \dem group, the crucial statement will be the existence of a short exact sequence of modules

  \[\xymatrix{0 \ar[r] & \omega_{-1}(\mathbf{F}_{p}) \ar[r] & \omega_2(\F_p) \ar[r] & J^* \ar[r] & 0 } \, , \tag{*}\]
  which will enable us to compute the cohomology groups $H^{i}(E_1,\alpha_{*}(J^*))$ for every $\pi$-point $\alpha$. The precise description of those modules will be detailed hereafter (see \ref{subsec-ses}.) Indeed it is worth noting that $\alpha_{*}\colon\mathfrak{mod}(\F_p G)\longrightarrow\mathfrak{mod}(\F_p E_1)$ induces a functor of triangulated categories between $\underline{\mathfrak{mod}}(\F_p G)$ and $\underline{\mathfrak{mod}}(\F_p E_1)$, since it is flat.
  
It is also noteworthy that the category of constant Jordan type modules is stable under Heller shifts: to be more precise, we can state the following theorem:

\begin{thm}A module $M$ is of constant Jordan type, if and only if any module which is stably isomorphic to $\Omega(M)$ is of constant Jordan type. \label{HelCst}
\end{thm} 

A proof can be found -as for everything dealing with modules of constant Jordan type- in \cite{JordanType}. Since the trivial module $\F_p$ is obviously of constant Jordan type, the modules $\omega_2 (\F_p)$ and $\omega_{-1}(\F_p)$ which appear in the previous exact sequence are also of constant Jordan type: even better, we know their Jordan type, such as stated in the following remark.

\begin{rem} Every module which is stably isomorphic to $\omega_2(\F_p)$ is of constant Jordan type and its stable Jordan type is $[1]$ (see \cite[Theorem 5.6]{JordanType}). \label{r1} \end{rem} 

For what concerns the module $\omega_{-1}(\F_p)$, it is already known to the reader: as it shall be proved later, it is $I(G)^*$ whose stable constant Jordan type is $[p-1]$ (see \ref{Ex2}).

As pointed out earlier, for an object $M$ in $\underline{\mathfrak{mod}}(\F G)$, there are many modules in $\mathfrak{mod}(\F G)$ whose equivalence class in the stable category is isomorphic to $\Omega(M)$. Yet, there exists a module without any projective summand in it and which is stably isomorphic to $\Omega(M)$. By a very slight abuse of notation, We denote such a (\emph{usual}) module verifying \emph{both} conditions $\Omega(M)$: indeed asking the absence of projective summand ensures the unicity of a representative (see\cite[p.14]{CarModAlg}). To be precise, in this article if we speak about the module (and not the stable module) $\Omega^1 (\F_p)$ we refer to the augmentation ideal of $G$ written $I(G)$ (and $\Omega^{-1}(\F_p)$ its dual), moreover $\Omega^2 (\F_p)$ will be the kernel of the application
\[\fonctionv{(\F_p G)^{|G|}=\langle e_{g},g\in G \rangle}{I(G)}{e_{g}}{X_{g}=g-1} \, .\]
We will encounter this module later (see \ref{P3}).

\subsection{A first half of the theorem}

Now we can prove half of the theorems, which is in fact a simple rephrasing of a a well-known theorem, so well-known among the specialists, that it is hard to know who should take credit for it. A reader interested by this case in the history of a theorem could refer to one article from {\sc Magnus} (\cite{Mag}), a note on this article by {\sc Blackburn} (\cite{Bl}), a letter in {\sc Serre}'s correspondence and an article of {\sc Gaschütz} (\cite{Gas}).

\begin{prop}Let $\mathcal{F}_n$ be the free pro-$p$-group of rank $n$ and $\tilde{\mathcal{H}}$ be a closed subgroup of $\mathcal{F}_n$. Let $G=\mathcal{F}_n/\tilde{\mathcal{H}}$, then $\tilde{\mathcal{H}}/\Phi(\tilde{\mathcal{H}})$ is stably isomorphic as a $G$-module to $\Omega^2(\F_p)$. \label{GB} \end{prop}

A conceptual proof of this fact can be found in \cite{Fri} and a more down-to-earth using Fox derivatives is implicitly present in \cite{Bl}. Here, the implicitly means that the reader has just to add "pro-$p$" every where it makes sense and read "Frattini subgroup" instead of "derived subgroup".

This proves the main theorem when $\mathcal{G}_{\k}(p)$ is a free pro-$p$-group. Indeed, by definition

\[\hat{H}^{s}(G,\Omega^2 (\F_{p}))\simeq\hat{H}^{s-2}(G,\F_p)\, ; \]
moreover such module is of constant Jordan type and its stable Jordan type is $[1]$, as pointed out in the remark above (in \ref{r1}). Only one piece of information is missing: the dimension of the fixed points of $(\tilde{\mathcal{H}}/\Phi(\tilde{\mathcal{H}}))^*$ under the action of $G$. Let us take a closer look at the five term exact sequence (\cite[Cor. 2.4.2]{NSW}):
\[\xymatrix{0\ar[r] & H^1 (G,\F_p) \ar[r] & H^1(\mathcal{F}_n,\F_p) \ar[r] & H^1 (\mathcal{H},\F_p)^{G} \ar[r] & H^2 (G,\F_p) \ar[r] & H^2 (\mathcal{F}_n , \F_p) }\, . \]
It is well known that $\hat{H}^2 (\mathcal{F}_n ,\F_p)=0$ (\cite{CoGal}) and that $H^1 (G,\F_p)$. Now, let us inspect $H^1 (\mathcal{H},\F_p)^G$, we have indeed
\[\begin{array}{rclr}
H^1 (\tilde{\mathcal{H}},\F_p)&\simeq&(\hom(\tilde{\mathcal{H}},\F_p))^{G}&(\F_p\text{ is a trivial module}) \\
 &\simeq& \hom(\tilde{\mathcal{H}}/\Phi(\tilde{\mathcal{H}}),\F_p)^{G} &\text{(by the property of the Frattini subgroup)}
\end{array} \]
This litany of isomorphisms will be often used in this paper. Therefore we can state the following equality
\[\dim H^0 (G,(\mathcal{H}/\Phi(\mathcal{H}))^{*})=d_{2}(G)+(n-d_{1}(G))\, . \]

We then need this basic lemma:
\begin{lm} Let~$G$ be a~$p$-group, and let $M$ be an $\F_p G$-module. Now set
\[n=\dim H^0 (G,M)-\dim \hat{H}^0(G,M) \, ,  \] then $M=M_1\oplus(\F_p G)^{n}$ where $M_1$ is an $\F_p G$-module without any projective summand. \label{Ev}\end{lm}

\begin{proof}
Let us decompose $M$ as
\[M=M_1\oplus  P\, , \]
where $P$ is a projective module and $M_1$ has no projective summand. In fact, we may suppose that $P=(\F_p G)^{m}$, since projective $\F_p G$-modules are free. We would like to show that $m=n$. To this end, let us remark
\[\hat{H}^0(G,M)=\hat{H}^0 (G,M_1)\, . \]
Since $\dim\hat{H}^0 (G,M_1)+m=\dim H^0 (G,M)$, it remains to prove that
\[\dim H^0 (G,M_1)=\dim \hat{H}^0 (G,M_1)\, . \]
If this were not the case, there would be an element $x\in M_1$ such that $x\cdot\No\neq 0$ (where $\No$ is the norm) so that there would be a projective summand in $M_1$ (see \cite[Lemma 1.31]{Gui}, which is absurd.
\end{proof}

Now, we can properly compute $H^0 (G,\mathcal{H}/\Phi(H))$. Indeed, since $\mathcal{H}/\Phi(\mathcal{H})$ is stably isomorphic to $\Omega^{2}(\F_p)$, its dual is stably isomorphic to $\Omega^{-2}(\F_p)$, hence
\[\dim H^{0}(G,(\mathcal{H}/\Phi(\mathcal{H}))^*)-\dim\hat{H}^0 (G,(\mathcal{H}/\Phi(\mathcal{H}))^*)=n-d_{1}(G) \, ,\]
where $d_{i}(G)=\dim_{\F_{p}}\hat{H}^{i}(G,\F_p)\, .$
 According to the previous lemma
\[(\mathcal{H}/\Phi(\mathcal{H}))^{*}=\Omega^{-2}(\F_p) \oplus(\F_p G)^{n-d_1 (G)}\, , \]
where $\Omega^{-2}(\F_p)$ is a module without any projective summand and stably isomorphic to $\Omega^{-2}(\F_p)$, hence its dual is isomorphic to $\Omega^{2}(\F_p)$. Since there is no projective summand

\[\dim H^0 (G,\Omega^{-2}(\F_p)^{*})=\dim \hat{H}^0 (G,\Omega^{-2}(\F_p)^{*})=\dim \hat{H}^{-2}(G,\F_p)=d_{1}(G)\, . \]

Now, we can easily conclude:
\[\dim H^0 (G,\mathcal{H}/\Phi(\mathcal{H}))=n-d_{1}(G)+\dim H^0 (G,\Omega^{-2}(\F_p)^*) \]

Hence we obtain the expected result:

\[\dim H^0 (G,\mathcal{H}/\Phi(\mathcal{H}))=n\, . \]

Alternatively, we could have used \cite[Satz 2]{Gas}.

It now remains to confront the case where $\mathcal{G}_{\k}(p)$ is a \dem group.

\section{Proof of the main theorems}\label{P2}

In this section, as in the previous one, $p$ is a prime number, $\k$ is a local field and $\mathcal{G}_{\k}(p)$ is the Galois group of a maximal pro-$p$-extension of $\k$. From now on, we assume that $\mathcal{G}_{\k}(p)$ is a \dem group. In order to simplify our proofs, we will assume below that $p\neq 2$, unless we explicitly write otherwise. All results in this section do still remain true when $p=2$, with some slight changes which we will indicate along the text. Let us fix then $n$ and $k$ such that
\[\mathcal{G}_{\k}(p)=\mathcal{D}_{k,n}\, .  \]
We treat this as an equality rather than an isomorphism, which is tantamount to choosing generators for the group once and for all.

Our objective is the proof of Theorem \ref{MT1} and Theorem \ref{MT2}, in the case when $\xi_{p}\in\mathbf{k}$, so that the above group-theoretical hypotheses are in force. As previously mentioned, our key argument is a short exact sequence: we shall prove it first, and then draw the consequences from it. No mention of Galois theory will be made, since we have already translated the problem of studying $J(\K)=\K^{\times}/\K^{\times p}$ into a problem of group theory, for $\K/\k$ a finite $p$-Galois extension whose Galois group is $G$.

\subsection{The short exact sequence...} \label{subsec-ses}

As previously recalled, the \dem group $\mathcal{D}_{k,n}$ is but a quotient of the free pro-$p$-group on $n$ generators $x_1, \ldots, x_n$, denoted here $\mathcal{F}_{n}$, by the normal subgroup generated by $x_{1}^{p^k}(x_1,x_2)\ldots (x_{n-1},x_{n})$. This element will be denoted $\delta$ and we set
\[\pi\colon\mathcal{F}_{n}\longrightarrow\mathcal{D}_{k,n}\, , \] 
the canonical projection. If $p=2$, then $\delta$ should be changed, and we have to distinguish multiple cases (see \ref{subsec-demushkin}), however what follows remains true without any change.

Now, let us construct the epimorphism in the short exact sequence appearing in proposition \hyperlink{GT3}{D}, which means a map from a module $\omega_2 (\F_{p})$ (stably isomorphic to $\Omega^{2}(\F_p)$) onto $J^*$. Remember that $J^{*}=\mathcal{H}/\Phi(\mathcal{H})$ for a closed, normal subgroup $\mathcal{H}$ of $\mathcal{D}_{k,n}$. So we set
\[\tilde{\mathcal{H}}=\pi^{-1}(\mathcal{H}) \, . \]Note immediately that $\ker\pi\subset \tilde{\mathcal{H}}$, since $\ker\pi$ is the normal subgroup of $\mathcal{F}_{n}$ generated by $\delta$. Hence, there exists an $\F_p$-linear map from $\tilde{\mathcal{H}}/\Phi(\tilde{\mathcal{H}})$ onto $\mathcal{H}/\Phi(\mathcal{H})$. Moreover, this map is a morphism of $G$-modules.

Indeed, first, note that $\mathcal{F}_{n}/\Phi(\tilde{\mathcal{H}})$ is isomorphic to $G$: since $\ker\pi$ is equal to $\operatorname{Gr}(\delta)$ and $\delta$ lies in $\tilde{\mathcal{H}}$, we have that $\pi\colon\mathcal{F}_{n}\longrightarrow\mathcal{D}_{k,n}$ induces an epimorphism from $\mathcal{F}_{n}/\tilde{\mathcal{H}}$ onto $\mathcal{D}_{k,n}/\mathcal{H}$, and by definition of $\tilde{\mathcal{H}}$ it is obviously a monomorphism, hence it is an isomorphism of groups.

Furthermore, the induced map, which is written $\pi_{\mathcal{H}}$, is $G$-equivariant, since the action on those modules is but the action by conjugation. Thus $\pi_{\mathcal{H}}$ is an epimorphism of modules from $\tilde{\mathcal{H}}/\Phi(\tilde{\mathcal{H}})$ which is stably isomorphic to $\Omega^2 (\F_p)$, according to Proposition \ref{GB}, onto $J^*$. Since we have fixed an extension $\K/\k$, we have fixed the subgroup $\mathcal{H}$ (and consequently $\tilde{\mathcal{H}}$). Therefore from now on, when we speak of \emph{the} module $\omega_2 (\F_p)$ we mean $\tilde{\mathcal{H}}/\Phi(\tilde{\mathcal{H}})$ as an $\mathcal{F}_{n}/\tilde{\mathcal{H}}$-module, unless we say explicitly so. In order to prove the short exact sequence, it remains to study the kernel of $\pi_{\mathcal{H}}$: it is done in the following lemma.

\begin{lm}There exists a unique cyclic  $\F_p G$-module of dimension $|G|-1$ (up to isomorphism). It is stably isomorphic to $\Omega^{-1}(\F_p)$, and in fact isomorphic to \[ \langle r|r\cdot\No=0 \rangle\, ,\]where $\No$ denotes the norm
\[\No:=\sum\limits_{g\in G}g\, . \]
\end{lm}

\begin{proof}
Let us translate one by one the hypothesis of those lemma: suppose $M$ is a module verifying the conditions of the lemma; since $M$ is cyclic, there exists an epimorphism $\F_{p} G\longrightarrow M$. Because of the dimension of $M$, its kernel is of dimension 1, it is of course $\F_{p}$, both as a vector space and as a module. Therefore the following sequence is exact
\[\xymatrix{0 \ar[r] & \F_p \ar[r] & \F_p G \ar[r] & M \ar[r] & 0}\, , \]
which, by definition, means \[M\simeq\Omega^{-1}(\F_p)\, .\]
Now, note that the monomorphisms from $\F_p$ into $\F_p G$, are just the $f_{c}\colon1\mapsto c\No$, where $c\in\F_p$ and $\No$ is the norm. Hence the lemma.
\end{proof}

This module, which is simply $I(G)^*$, will be denoted by $\omega_{-1}(\F_p)$ in the rest of the paper.

Now, it is obvious that the kernel of the map from $\tilde{\mathcal{H}}/\Phi(\tilde{\mathcal{H}})$ onto $\mathcal{H}/\Phi(\mathcal{H})$ is monogenous: it is indeed generated by $[\delta]$, the class of $\delta$ modulo $\Phi(\tilde{\mathcal{H}})$. We then have to compute the dimension of $J^*$.

For every finitely generated pro-$p$-group $\mathcal{U}$, $d_1 (\mathcal{U})$ denotes the minimal number of topological generators of $\mathcal{U}$, or equivalently the dimension of $H^1 (\mathcal{U},\F_p)$ or the one of $\mathcal{U}/\Phi(\mathcal{U})$ (see \cite[4.2]{CoGal}).

According to \cite[Example 6.3]{KochL} and to \cite[Exercice 6 p.41]{CoGal}, the following formulae hold:
\begin{equation}\left\lbrace\begin{array}{rcl}
d_1 (\tilde{\mathcal{H}})&=&(\mathcal{F}_{n}:\tilde{\mathcal{H}})(n-1)+1\\
d_1 (\mathcal{H})&=&(\mathcal{D}_{k,n}:\mathcal{H})(n-2)+2
\end{array}\right. \label{Dim}
\end{equation}

Since we have that $(\mathcal{F}_{n}:\tilde{\mathcal{H}})=|G|=(\mathcal{D}_{k,n}:\mathcal{H})$, the dimension of $\ker\pi_{\mathcal{H}}$ is exactly $|G|-1$, therefore we can use the lemma and conclude. Thus the following proposition holds:

\begin{prop}The following sequence is exact:
\begin{equation} \xymatrix{0 \ar[r] & \omega_{-1}(\mathbf{F}_{p}) \ar[r]^{\kappa} & \omega_2(\F_p) \ar[r]^{\pi_{\mathcal{H}}} & J^* \ar[r] & 0 }\, . \label{SE}
\end{equation} \end{prop}

\begin{rem}It should be pointed out that, in the previous exact sequence, $\omega_{2}(\F_p)$ does not necessarily verify the minimality condition we have set, whereas $\omega_{-1}(\F_p)$ always does so. \end{rem}

From now on, we fix $\kappa$ to be the map in the previous short exact sequence: it will play a major role, before disappearing at the end of this section.

\subsection{...and its consequences}\label{Cnsq}

Now, we are in possession of the required tools to show the main theorems. We will have to distinguish between two cases, according as the morphism $\kappa$ in \eqref{SE} is stably zero or not. In section \ref{P4}, we shall give examples of both cases, showing they actually occur.

\subsubsection{When $\kappa$ is stably zero}

In $\underline{\mathfrak{mod}}(\F_{p} G)$ the triangle
\[\xymatrix{\Omega^{-1}(\F_p)\ar[r]^{\kappa} & \Omega^{2}(\F_p) \ar[r] & J^* \ar[r]& \Omega^{-2}(\F_p)}\, , \]
is distinguished according to \cite[3.1.]{TrCat}

If $\kappa$ is stably zero, then, according to \cite[4.4.]{TrCat}, in the stable category the following isomorphism stands
\[J^* \simeq\Omega^{2}(\F_p)\oplus\Omega^{-2}(\F_p) \, .  \]
Thus $J^*$ is of constant Jordan type and its stable Jordan type is $[1]^2$, without any condition on the group $G$. Furthermore, since $\kappa=0$ stably, the beginning of the long exact sequence of cohomology is but
\[\xymatrix{0\ar[r] & H^0 (G,\omega_{-1}(\F_p))\ar[r] & H^0 (G,\omega_{2}(\F_p))\ar[r] &  H^0 (G,J^* (\K))\ar[r] &  H^1 (G,\omega_{-1}(\F_p))\ar[r] & 0}\, , \]
hence we get
\[\dim H^0 (G,J(\K)^*)=d_{2}(G)+(n-d_{1}(G)) \, . \]

We have therefore shown a little more that we announced:

\begin{prop}If the map $\kappa$ is zero, then $J(\K)$ is isomorphic in the stable module category to
\[J(\K)\simeq \Omega^2 (\F_p)\oplus\Omega^{-2}(\F_p)\, . \]Hence $J(\K)$ is of constant Jordan type $[1]^2$ and the cohomology groups of $G$ with coefficients in $J(\K)$ are in fact
\[\left\lbrace\begin{array}{rcl}
\hat{H}^{s}(G,J(\K))&\simeq&\hat{H}^{s+2}(G,\F_p)\oplus\hat{H}^{s-2}(G,\F_p) \\
H^0 (G,J(\K))&\simeq&\F_{p}^{d_{2}(G)+d_{1}(G)+(n-2d_{1}(G))}

\end{array} \right.\]\label{KappaZero}
  \end{prop}

So we have proved (2)(a) of Theorem~\ref{MT1}, and a little more than the second statement of Theorem~\ref{MT2} under the current assumption on~$\kappa $. Remark that we did not need to make further assumptions on $G$ in order to prove this proposition: it is an improvement of both main theorems.

Let us discuss a bit more the structure of $J(\K)$. We certainly know that in the category of (not stable) modules, we have 
\[J(\K)\simeq\Omega^{2}(\F_p)\oplus\Omega^{-2}(\F_p)\oplus P\, , \]
where $P$ is a projective module. Bear in mind that we write $\Omega^{k}(\F_p)$ for a module which is stably isomorphic to $\Omega^{k}(\F_p)$ and not containing any projective summand. It is possible to compute the number of copies of $\F_p G$ contained in $P$:

\begin{prop}When $\kappa$ is zero, the $G$-module $J(\K)$ can be decomposed in the following way
\[J(\K)\simeq\Omega^{2}(\F_p)\oplus\Omega^{-2}(\F_p)\oplus (\F_p G)^{n-2d_{1}(G)}\, . \] \end{prop}

\begin{proof}As previously proved, we have a stable isomorphism
\[J(\K)\simeq \Omega^2 (\F_p)\oplus\Omega^{-2}(\F_p) \, , \]
it is therefore sufficient to compute the number of copies of the free module in $J(\K)$. Using lemma \ref{Ev}, we get that this number is equal to
\[\dim H^0(G,J(\K))-\dim \hat{H}^0(G,J(\K))=d_{2}(G)+n-d_{1}(G)-d_{2}(G)-d_{1}(G)\, ,\]
because we have clearly set that

\[\begin{array}{rclcl}
H^0 (G,\Omega^{2}(\F_p))&\simeq& \hat{H}^0 (G,\Omega^{2}(\F_p))&\simeq&\F_{p}^{d_{1}(G)}\\
H^0 (G,\Omega^{-2}(\F_p))&\simeq& \hat{H}^0 (G,\Omega^{-2}(\F_p))&\simeq&\F_{p}^{d_{2}(G)}\end{array} \]
Which concludes the proof.
\end{proof}

It should be remarked that this proposition gives us a necessary -but not sufficient- condition on $G$ in order for $\kappa$ to be zero: as it is quite convenient, we promote this small remark to a corollary.

\begin{cor}If $\kappa$ is stably zero, then $2d_{1}(G)\leq n$.\end{cor}

For example, note that if $\K/\k$ is the {\em maximal} $p$-elementary abelian extension, then $n=d_1(G)$, and according to the previous corollary, it is not possible for $\kappa$ to be stably zero;  hence $\K^{\times p}/\K^{\times}$ is never stably isomorphic to $\Omega^{2}(\F_p)\oplus\Omega^{-2}(\F_p)$, as will be clear from the study of the alternative case, to which we turn now.

\subsubsection{When $\kappa$ is stably non-zero}

If $\kappa$ does not vanish in the stable module category, we will proceed in two steps: first, we shall draw the consequences of the short exact sequence \eqref{SE} in cohomology, secondly we shall prove that $J(\K)$ is of constant Jordan type. Let us suppose from now on that $\kappa$ is not stably zero.

\begin{prop}We have the following equality
\[\dim H^0 (G,J(\K))=d_{2}(G)+(n-d_{1}(G))-1\, . \]
 \end{prop}
 
\begin{proof}
Note that in the long exact sequence in cohomology, the map
\[H^1(G,\omega_{-1}(\F_p))\longrightarrow H^1(G,\omega_{2}(\F_p)) \]
is not zero. Indeed, remark that the following diagram is commutative
\[\xymatrix{H^1(G,\omega_{-1}(\F_p))\ar@{=}[d]\ar[r] & H^1(G,\omega_{2}(\F_p)) \ar@{=}[d] \\ \hat{H}^{2}(G,\F_p) \ar[r]^{\smile [\kappa]}& \hat{H}^{-1}(G,\F_p)
} \]
However, by Tate duality \cite[\S VI.7]{Brown}, the pairing given by
\[\fonctionv{\hat{H}^{i}(G,\F_p)\otimes\hat{H}^{-i-1}(G,\F_p)}{\hat{H}^{-1}(G,\F_p)\simeq\F_p}{a\otimes b}{a\smile b} \]
is non-degenerate. Therefore the previous map is non-zero, and consequently it is an epimorphism, for $H^1 (G,\omega_{2}(\F_p))\simeq\F_p$. Thus the beginning of the long exact sequence of cohomology is just but

\[\xymatrix{0 \ar[r] & H^0(G,\F_p) \ar[r] & H^0(G,\omega_2 (\F_p))\ar[r] & H^0(G,\F_p) \ar[r] & H^2(G,\F_p)\ar[r] & \hat{H}^{-1}(G,\F_p) \ar[r] & 0} \]
Hence we obtain the announced equality.
\end{proof}

Now, we have to make further assumptions in order to compute the cohomology groups.

From now on, in this section $G$ is a $p$-group such that $H^{\bullet}(G,\F_p)$ is a Cohen-Macaulay ring: since this hypothesis is -as far as we know- quite uncommon among the literature in Galois theory, we shall sometimes recall this hypothesis in order to lay emphasis on it.

\begin{lm}\label{LCent}Let $s,j\in\Z$ such that $j-s > 0$ and let $\gamma\colon\omega_{s}(\F_{p})\longrightarrow\omega_{j}(\F_{p})$ a map of modules, where $\omega_{s}(\F_p)$ and $\omega_{j}(\F_p)$ are \emph{any} modules stably isomorphic to $\Omega^{s}(\F_p)$ and $\Omega^j (\F_p)$. Then consider the distinguished triangle in the stable module category
 \begin{equation}\xymatrix{ \Omega^{s}(\F_{p}) \ar[r]^{\gamma} & \Omega^{j}(\F_{p}) \ar[r] & C_{\gamma} \ar[r] & \Omega^{s-1}(\F_p) }\, . \label{SE1} 
 \end{equation}
 If $G$ is such that $H^{\bullet}(G,\F_p)$ is a Cohen-Macaulay ring, then in the long exact sequence of cohomology the following maps 
\[\hat{H}^{l}(G,\omega_{s}(\F_{p}))\longrightarrow\hat{H}^{l}(G,\omega_{j}(\F_{p})), \quad l\geq j+1 \quad\text{or}\quad l\leq s-1 \] are zero.  \end{lm}

\begin{proof}Three things shall be remembered. First, bear in mind that $\omega_{l}(\F_p)$ denotes \emph{a} module which is stably isomorphic to \emph{the} object $\Omega^{l}(\F_p)$ in the stable category. 

Secondly, for every pair of integers $n_1,n_2$, there exists an isomorphism between $\hat{H}^{n_1}(G , \F_p )$ and $\underline{\hom}(\Omega^{n_1 + n_2}(\F_p),\Omega^{n_2}(\F_p))$. If $\kappa$ is an element of $\underline{\hom}(\Omega^{n_1+n_2}(\F_p),\Omega^{n_2}(\F_p))$, by $[\kappa]$ we mean the corresponding class of cohomology in $\hat{H}^{n_1}(G , \F_p)$.

Thirdly, if $a,b$ are two cohomology classes respectively of degree $n_1$ and $n_2$ such that $a=[\alpha]$ and $b=[\beta]$ where $\alpha\colon\Omega^{n_1+n_2+n_3}(\F_p)\longrightarrow\Omega^{n_2+n_3}(\F_p)$ and $\beta\colon\Omega^{n_2 + n_3}(\F_p)\longrightarrow\Omega^{n_3}(\F_p)$, then $a\smile b=[\alpha\circ\beta]$. (\cite[\S 4.5]{CoRi})

Now, the short exact sequence \eqref{SE1} gives birth to a long exact sequence in the stable-module category (\cite[Prop. 4.2]{TrCat}), which is but the long exact sequence in cohomology. Let us take a closer look to it:

\[\xymatrix{\ldots\ar[r] &\underline{\hom}(\F_p,\Omega^{s-l}(\F_p))\ar[r]^{m_{\gamma}} &\underline{\hom}(\F_p,\Omega^{j-l}(\F_p))\ar[r] & \underline{\hom}(\F_{p},\Omega^{-l}(C_{\gamma}))\ar[r] & \\ \ar[r] & \underline{\hom}(\F_p,\Omega^{s-1-l}(\F_p)) \ar[r] & \underline{\hom}(\F_p,\Omega^{j-1-l}(\F_p)) \ar[r] & \ldots  &  } \]

However the application defined by
\[\fonction{m_{\gamma}}{\underline{\hom}(\F_p,\Omega^{s-l}(\F_p))}{\underline{\hom}(\F_p,\Omega^{j-l}(\F_p))}{f}{f\circ\gamma} \]
is simply the cup product by $[\gamma]\in\hat{H}^{s-j}(G,\F_p)$ from $\hat{H}^{l-s}(G,\F_p)$ to $\hat{H}^{l-j}(G,\F_p)$. Now, it is known \cite[Thm. 3.1 and lemma 2.1]{BensProd}, under the assumption that $H^{\bullet}(G,\F_p)$ is a Cohen-Macaulay ring, that such cup products are zero, as soon as $l <s$ (\cite[Thm. 3.1]{BensProd}) or $l> j$ (\cite[Lemma 2.1]{BensProd}). Hence we get the proposition.

 \end{proof}
 
 \begin{rem}If $p=2$, in order to apply \cite[Thm. 3.1 and lemma 2.1]{BensProd}, we have in addition to suppose that $G$ is different from the quaternion groups. \end{rem}

\begin{cor} Let $G$ be as in the lemma. If $M$ is an $\F_p G$-module verifying the following exact sequence \begin{equation}\xymatrix{0 \ar[r]& \omega_{-1}(\F_{p}) \ar[r] & \omega_{2}(\F_{p}) \ar[r] & M \ar[r] &  0}\, , \label{SEPrime}
\end{equation}then the Tate cohomology groups with coefficients in $M$ verify
\[\hat{H}^{s}(G,M)=\hat{H}^{s+2}(G,\F_{p})\oplus\hat{H}^{s-2}(G,\F_{p}) \quad \forall s,\, s\geq 2,\,\text{or}\, \, s\leq -3 \, . \]\end{cor}

\begin{proof}
It is sufficient to take a look on the long exact sequence in cohomology and apply the previous lemma.
\end{proof}

Therefore, we have already proved another big part of our theorem: the computation of the cohomology groups of degree higher than $2$ (and lower than $-4$) is done and conform to what was announced in case (2)(b) of Theorem~\ref{MT1}. Let us now address the case of the first cohomology group.

\begin{prop}The groups $H^1 (G, J^*)$ and $H^3 (G,\F_p )$ are isomorphic.
 \end{prop}
 
 \begin{proof}Let us again look at the long exact sequence in cohomology. Remember that
 \[H^1 (G,\omega_{-1}(\F_p))\longrightarrow H^1(G,\omega_{2}(\F_p)) \]
 is an epimorphism according to Tate duality, since we have supposed that $\kappa$ is not stably zero (\cite[\S VI.7]{Brown}). Therefore we may write
\[\xymatrix{0\ar[r] & H^1 (G,J(\K))\ar[r] & H^2(G,\omega_{-1}(\F_p)) \ar[r] &H^2(G,\omega_{2}(\F_p))\ar[r] &\ldots} \] 
   and as stated in lemma \ref{LCent} the arrow
 \[H^2(G,\omega_{-1}(\F_p))\longrightarrow H^2(G,\omega_{2}(\F_p)) \, ,\]
is zero, hence the long exact sequence in cohomology gives us
\[\xymatrix{0\ar[r] & H^1 (G,J(\K))\ar[r] & H^2(G,\omega_{-1}(\F_p)) \ar[r] &0} \]
which concludes the proof.
 
 \end{proof}

Now  the proof of Theorem~\ref{MT1} is finally complete, in all cases. It is time to address the proof of Theorem \ref{MT2}. Let us recall a proposition due to  {\sc Benson} \cite[Proposition 8.12.1]{BensL}

\begin{prop}Let $\alpha$ be a $\pi$-point and $G=E_{k}$, where $k\geq 2$. If $\zeta\in\hat{H}^{-l}(E_k,\F_p)$ with $l>0$, then $\res_{\alpha}(\zeta)$ is zero.\end{prop}

Now, we can prove the promised theorem.

\begin{thm}Let $\mathbf{K}/\mathbf{k}$ be an elementary abelian $p$-extension. If $\mathbf{K}/\mathbf{k}$ is not cyclic, then $J(\mathbf{K})^{*}=\mathbf{K}^{\times}/\mathbf{K}^{\times p}$ is a $Gal(\mathbf{K}/\mathbf{k})$-module of constant Jordan type. Furthermore, its stable Jordan type is $[1]^2$.  \end{thm}

\begin{proof}We recall that if $M$ is an $E_1$-module, $n_{j}(M)$ denotes the number of blocks of size $j$ in the decomposition of $M$. Therefore our goal is to prove that $n_{1}(\alpha_* (J^*))=2$ and $n_{p}(\alpha_* (J^*))=(n-2)\frac{|G|}{p}$ for every $\pi$-point $\alpha$.

Now, remember that in the exact sequence
\[\xymatrix{0 \ar[r] & \omega_{-1}(\F_p) \ar[r]^{i} & \omega_2 (\F_p) \ar[r] & J^* \ar[r] & 0 } \, , \]
the map $i$ is in fact a cohomology class in $\hat{H}^{-3}(E_k ,\F_p)$. We have previously remarked that in the long exact sequence of cohomology the maps
\[\hat{H}^{l}(G, \omega_{-1}( \F_p))\longrightarrow\hat{H}^{l}(G,\omega_{2}(\F_p)) \quad \forall l\in\Z \, , \]
is the cup product by $[i]\in\hat{H}^{-3}(E_k , \F_p)$.

So, let $\alpha$ be a $\pi$-point. Since $\res_{\alpha}([i])$ is zero according to the previous proposition, numerous morphisms are zero in the long exact sequence in cohomology. To be more precise, it leads to the following short exact sequence for all $l\in\Z$:
\[\xymatrix{0 \ar[r] & \hat{H}^{l}(E_1, \alpha_{*}(\omega_{2}(\F_p))) \ar[r] & \hat{H}^{l}(E_1, \alpha_{*}(J^*)) \ar[r] & \hat{H}^{l+1}(E_1, \alpha_{*}(\omega_{-1}(\F_p))) \ar[r] & 0} \, . \]
Since a $\pi$-point induces a morphism of triangulated category, $\alpha_{*}(\Omega(M))=\Omega(\alpha_{*}(M))$, so that we can compute the leftmost and rightmost groups in the previous short exact sequence. Hence we deduce this precious piece of information:
\[H^1(E_1,\alpha_{*}(J^*))=\hat{H}^1(E_1,\alpha_{*}(J^*))= \F_{p}^2 \, . \]

Thus we know, that for every $\pi$-point $\alpha$, $\alpha_{*}(J^*)$ has in its decomposition exactly two blocks whose size is not $p$. It remains to prove that their size is exactly 1. To this end, let us compute the dimension of $H^0 (E_1,\alpha_{*}(J^*))$.

Knowing the nullity of the map $H^2(E_1,\alpha_{*}(\Omega^{-1}(\F_p)))\longrightarrow H^2 (E_1,\Omega^2 (\F_p))$, the long exact sequence in cohomology gives us the following exact sequence:
\[\xymatrix{0 \ar[r] & H^0(E_1,\alpha_{*}(\omega_{-1}(\F_p))) \ar[r] & H^0 (E_1,\alpha_{*}(\omega_{2}(\F_p))) \ar[r] & H^0 (E_1,\alpha_{*}(J^*)) \ar[d] \\  &  & 0 & H^1 (E_1,\alpha_{*}(\omega_{-1}(\F_p))) \ar[l]^{a_1} } \]

Since we know both the dimension of $\omega_{-1}(\F_p)$ and $\omega_2 (\F_p)$ and their stable constant Jordan type (resp. $[p-1]$ and $[1]$), we deduce
\[\left\lbrace\begin{array}{rcl}
\dim H^0(E_1 ,\alpha_{*}(\omega_{-1}(\F_p))&=&\frac{|G|}{p} \\
\dim H^0(E_1 ,\alpha_{*}(\omega_{2}(\F_p))&=&(n-1)\cdot\frac{|G|}{p}+1 \\
\dim H^1(E_1 ,\alpha_{*}(\omega_{-1}(\F_p))&=& 1
\end{array}\right. \] 

Injecting these piece of information in the previous exact sequence we obtain

\[\dim H^0 (E_1,\alpha_{*}(J^*))=1+((n-1)\frac{|G|}{p}+1-\frac{|G|}{p})=2+(n-2)\frac{|G|}{p} \, , \]
hence we get the following equality:
\[n_p (\alpha_{*}(J^*))=(n-2)\frac{|G|}{p} \, . \]

As previously remarked, we know that in the decomposition there are two blocks whose size is not $p$. Let $\ell_1$ and $\ell_2$ be their size; of course, we have
\[\dim J^*=n_{p}(\alpha_{*}(J^*))\cdot p + \ell_1 +\ell_2 \, . \]
Because we already know the dimension of $J^{*}$ (see \ref{Dim}), this leads to 
\[2=\ell_1 + \ell_2 \, . \]
We deduce that $\ell_1 = \ell_2 =1$, hence $n_1 (\alpha_* (J^*))=2$, as expected.
 \end{proof}

This completes the proof of Theorem~\ref{MT2}.

\begin{rem} A peculiar case is worth noting: if the absolute Galois group has $p$-rank two, the module $J^*$ is of dimension $2$ according to the previous computation of the dimension, and since it is of constant Jordan type $[1]^2$, it is simply isomorphic as a module to $\F_p \times\F_p$!

Furthermore such case is not a pathological made-up one, indeed consider $\k=\Q_{\ell}(\xi_p)$, with $\ell\neq p$. In this case, according to \cite[Theorem 4.8]{Gui}, $\k^{\times}/\k^{\times p}$ is of dimension $2$, hence $\mathcal{G}_{\k}(p)$ is generated by two elements. \label{Remarque}

\end{rem}

\subsubsection{On the vanishing condition}
The condition regarding $\kappa$, as natural as it can be in this text, is not easy to check nor to express in few words, therefore we will try to find an equivalent condition and give criteria in order to distinguish between the two cases.

\begin{prop}The map $\kappa\colon\omega_{-1}(\F_p)\longrightarrow\omega_{2}(\F_p)$ is zero if and only if the following inflation map
\[\inf\colon H^2 (Gal(\K/\k),\F_p)\longrightarrow H^2 (\mathcal{G}_{k}(p),\F_p)\, , \]
is also zero. \end{prop}

\begin{proof}Let us look at the five term exact sequence associated to
\[\xymatrix{1 \ar[r] & \mathcal{H} \ar[r] & \mathcal{G}_{\k}(p) \ar[r] & G \ar[r]& 1}\, . \]
 In this case we obtain
\[\xymatrix{0\ar[r] & H^1 (G,\F_p) \ar[r] & H^1(\mathcal{G}_{\k}(p),\F_p) \ar[r] & H^1 (\mathcal{H},\F_p)^{G} \ar[r] & H^2 (G,\F_p) \ar[r] & H^2 (\mathcal{G}_{\k}(p) , \F_p) }\, . \]

Remember here that, $H^1(\mathcal{G}_{\k}(p))=\F_{p}^{n}$ and that $H^1 (\mathcal{H},\F_p)^{G}\simeq H^{0}(G,J)$. Now, let us suppose that $\kappa=0$, in this case $\dim H^0(G,J)=d_{2}(G)+d_{1}(G)+(n-2d_{1}(G))$ (according to Proposition \ref{KappaZero}). Hence by injecting these piece of information, we have in fact that the inflation
\[\inf\colon H^2(G,\F_p)\longrightarrow H^2 (\mathcal{G}_{\k}(p),\F_p) \]
is zero. 

The converse is similar.

 \end{proof}

\begin{rem}The previous condition echoes to the one introduced by {\sc Min\'a\v{c}}, {\sc Swallow} and {\sc Topaz} in \cite[Theorem 2]{MiSwTo}. \end{rem}

\section{A closer look at the maximal $p$-elementary abelian extension}\label{P3}

Let us resume the notation of the previous section : $\k$ is a local field, $p$ is a prime number, $\mathcal{G}_{\k}(p)$ is the Galois group of a maximal pro-$p$-extension of~$\k$ ; the field~$\K$ will now be specialized to be the maximal $p$-Kummer extension of~$\k$, a crucial particular case for which we can be much more precise than in general. Observe that~$\K$ is in Galois correspondence with~$\Phi (\mathcal{G}_{\k}(p))$, the Frattini subgroup ; the Galois group~$G= Gal(\K/\k)$ is elementary abelian of rank~$n$, and this number is also the number of generators for the \dem group $\mathcal{G}_{\k}(p)$.

Recall that many of our arguments are related to the existence of a short exact sequence of~$\F_pG$-modules
\[ 0 \longrightarrow \omega_{-1}(\F_p) \stackrel{\kappa }{\longrightarrow } \omega_2(\F_p) \longrightarrow J^* \longrightarrow 0 \, ,  \]
constructed in \S\ref{subsec-ses}. Moreover, as a ``model'' for~$\Omega^2(\F_p)$ we have in fact used the module
\[ M_{n} :=\Phi(\mathcal{F}_{n})/\Phi^{(2)}(\mathcal{F}_{n}) \, . \]
(Note that the group $\tilde{\mathcal{H}}$ appearing in \S\ref{subsec-ses} is now~$\Phi (\mathcal{F}_n)$, with our choice of field~$\K$, as is readily seen.) We shall now describe~$M_n$ in much more detail, as well as the map~$\kappa $, and as an application, we shall deduce several invariants of the module~$J^*$. (All of this will be conducted in the language of groups, and the field-theoretic notation will not appear.)

In more detail, we start the section by giving a presentation of~$M_n$, as an~$\F_pG$-module, by generators and relations (this part is rather technical). Next, this is used to give a concrete description of the morphism~$\kappa $. Finally we give a presentation of~$J^*$, again by generators and relations. We close the section by applying all this material in order to extend some results of {\sc Adem, Gao, Karagueuzian} and {\sc Min\'a\v{c}}. Indeed, in \cite{MKEG}, these authors study the module~$J^*$ under the current hypothesis on~$\K$, in the case~$p=2$; by the bye, they point out that most of their results can be extended to the $p$ odd case. Yet some of their examples and formulae should be slightly changed in order to remain true: our goal here is to show how to proceed.

We start our exposition by assuming that~$p>2$. The very minor modifications needed to deal with~$p=2$ will be given afterwards (see \ref{A1}).

\subsection{The module~$M_n$} \label{subsec-module-Mn}

\subsubsection{Notation \& conventions.}
If $\mathcal{G}$ is a finitely generated pro-$p$-group, $\Phi(\mathcal{G})$ denotes its Frattini subgroup, which means that $\Phi(\mathcal{G})=\mathcal{G}^p (\mathcal{G},\mathcal{G})\, .$ By $(\mathcal{G},\mathcal{G})$ we mean of course the group generated by the commutators
\[(g_1,g_2)=g_{1}^{-1}g_{2}^{-1}g_1 g_2 \, , \forall g_1,g_2\in \mathcal{G} \, .\]

Whenever $\mathcal{H}\triangleleft \mathcal{G}$, the group $\mathcal{G}$ acts by conjugation on $\mathcal{H}$, and we write
\[h^{g}=g^{-1}hg,\quad\forall h\in \mathcal{H},\forall g\in \mathcal{G} \, . \]

Thus  $\mathcal{G}$ acts on $M_\mathcal{G}= \Phi(\mathcal{G})/\Phi^{(2)}(\mathcal{G})$ by conjugation, and since the action of $\Phi(\mathcal{G})$ is trivial modulo $\Phi ^{(2)}(\mathcal{G})$,  we will study the action of $\mathcal{G}/\Phi (\mathcal{G}) \cong E_r$ for some~$r$.   As $M_\mathcal{G}$ is an $\F_p$ vector space, it is, all in all, an $\F_p E_r$-right module, with the elementary abelian group~$E_r$ identified as above.

On~$M_\mathcal{G}$ we shall use an additive notation, i.e. we write
\[[\alpha \beta ]=[\alpha ]+[\beta ] \, ,\forall \alpha ,\beta \in\Phi(\mathcal{G}) \, , \]
where $[\alpha]$ denotes the class of $\alpha$ modulo $\Phi^{(2)}(\mathcal{G})\, .$ However, usually the additive notation makes it unnecessary to use brackets, and we may simply write~$\alpha + \beta $ for~$\alpha , \beta  \in \Phi (\mathcal{G})$.

As for the action, our convention is to write~$[\alpha ] \cdot x$ for~$[\alpha^x]$ (where~$\alpha \in \Phi (\mathcal{G})$ and~$x \in \mathcal{G}$), and more generally we write~$[\alpha ] \cdot \lambda $ where~$\lambda \in \F_pE_r$. Moreover, we extend the convention we introduced in \S\ref{JT}: if we have used a letter, say~$x$, to denote an element of~$\mathcal{G}$, then we shall usually use the same letter~$x$ for its image in~$\mathcal{G}/\Phi (\mathcal{G})$ and the capitalized letter~$X$ for~$x - 1 \in \F_p [\mathcal{G}/\Phi (\mathcal{G})]$.

Here is an example of computation with all our conventions at work : 
\[ \alpha \cdot X = \alpha^x - \alpha = x^{-1} \alpha x \alpha^{-1} = (x, \alpha^{-1}) \, ,   \]
for~$\alpha \in \Phi (\mathcal{G})$ and~$x \in \mathcal{G}$.

This applies in particular to~$\mathcal{G}= \mathcal{F}_{n}$, the free pro-$p$-group on~$n$ generators. In this case we write~$M_n := \Phi (\mathcal{F}_{n}) / \Phi ^{(2)}(\mathcal{F}_{n})$.

\subsubsection{Some classical relations}

Let $\mathcal{G}$ be a finitely generated pro-$p$-group. Let us recall some classical formulae about commutators, translated into relations about~$\Phi (\mathcal{G}) / \Phi ^{(2)}(\mathcal{G})$ as a module with an action of~$\mathcal{G}$. When we specialize to~$\mathcal{G} = \mathcal{F}_{n}$ below, we shall see that we have in fact described {\em all} the relations, in the sense that we have a presentation.

\begin{lm}Let $x,y,z$ be three elements of $\mathcal{G}$, then the following relation holds in $\Phi^{(2)}(\mathcal{G})/\Phi(\mathcal{G})$:
\begin{equation}
(y,x)Z+(x,z)Y+(z,y)X=0 \, , \label{Jacobi}
\end{equation}where $X=x-1$ (similarly for $y$ and $z$). Furthermore we have:
\begin{equation}y^{p}\cdot X=(x,y)Y^{p-1}\, . \label{PP} \end{equation}
 \end{lm}

\begin{proof}
We recall the Hall-Witt formula (cf. \cite{DDMS} or \cite{Laz1}). Let $x,y,z$ be three elements of a pro-$p$-group $\mathcal{G}$, then
\[((x,y^{-1}),z)^{y}((y,z^{-1}),x)^{z}((z,x^{-1}),y)^{x}=1\, . \]

Indeed it is clear that
\[\begin{array}{rcl}
(x,y^{-1})^{y} &=&y^{-1}x^{-1}yxy^{-1}y \\
 &=&y^{-1}x^{-1}yx=(y,x) \, .
 \end{array} \]

We can deduce the following well-known relation, similar to the Jacobi relation in the realm of Lie algebras:
\[
(y,x)Z+(x,z)Y+(z,y)X=0 \pmod{\Phi^{2}(\mathcal{G})} \, .
\]

Indeed using the Hall-Witt relation and the previous remark, we have: \[\begin{array}{rcl}
1&=&((x,y^{-1}),z)^{y}((y,z^{-1}),x)^{z}((z,x^{-1}),y)^{x} \\
 &=&((x,y^{-1})^{-1}(x,y^{-1})^{z})^{y}((y,z^{-1})^{-1}(y,z^{-1})^{x})^{z}((z,x^{-1})^{-1}(z,x^{-1})^{y})^{x} \\
 &=&((y,x)^{-1}(y,x)^{z})((z,y)^{-1}(z,y)^{x})((x,z)^{-1}(x,z)^{y}) \, .
\end{array} \]

The following equalities, which could be found in \cite{DDMS}, will be useful:
 
\begin{equation}\left\{ 
\begin{array}{rcl}
 (x,yz)& = &(x,z)(x,y)^{z} \\
 (xy,z)& = &(x,z)^{y}(y,z) \\
 (y^{n},x) &=&(x,y)^{y^{n-1}}(x,y)^{y^{n-2}}\ldots (x,y) \, ,
\end{array} 
\, \right.\label{Com1}\end{equation}
hence \[(y^{k},x)=\sum_{i=0}^{k-1}(x,y)\cdot  y^{i} = (x,y)\cdot \sum_{i=0}^{k-1}  y^{i} \pmod{\Phi^{2}(\mathcal{G})} \, . \]
Since in $\F_p[T]$ the following polynomial identity holds
\[\sum_{i=0}^{p-1} T^{i}=(T-1)^{p-1}\, , \]
we get for~$k=p$ the following formula: 
\[ (y^p, x) = (x,y) \cdot Y^{p-1} \, .   \]
Given that 
\[ y^p \cdot X = (y^p)^x - y^p = x^{-1}y^p x - y^p = y^{-p}x^{-1}y^p x = (y^p, x)  \, ,   \]
we obtain the expected relation: \[y^{p}\cdot X=(x,y)Y^{p-1}\, . \qedhere \]
\end{proof}

\subsubsection{The free group} \label{subsec-Mn-basics}
Now we specialize to~$\mathcal{G}= \mathcal{F}_n$, the free pro-$p$ group on~$n$ generators, which will be called $x_1, \ldots, x_n$. The images of these in~$E_n = \mathcal{F}_n / \Phi (\mathcal{F}_n) $ will also be called~$x_1, \ldots, x_n$. We write~$X_i = x_i - 1 \in \F_p E_n$.

According to the previous relations \eqref{Com1}, the $2$-commutators (i.e. the $(x_{i},x_{j})$) and the $x_{i}^{p}$ form a generating system for $M_{n}$ as $\F_p E_{n}$-module. The first thing we note is that 
\[ (x_i, x_j ) = - (x_j, x_i) \, .   \]
Simply because~$x_i$ commutes with $x_i^p$, we certainly have 
\[ x_i^p \cdot X_i = 0 \, .   \]
Next, from the relation \eqref{PP} of the lemma, we have 
\[ x_j^p \cdot X_i = (x_i, x_j) X_j^{p-1} \, .   \]
And finally, from \eqref{Jacobi}, we obtain :
\begin{equation}
(x_k, x_j) \cdot X_i + (x_j, x_i) \cdot X_k + (x_i, x_k) \cdot X_j = 0 \, . 
\end{equation}

Ultimately, we shall prove that the four types of relations just given between the generators provide a presentation for~$M_n$, ie they generate the module of relations.

The strategy is as follows. First we note that it is enough to include the 2-commutators with $i<j$, of course, so we have ${n \choose 2}$ commutators and $n$ elements of the form $x_{i}^{p}$. Let~$F_{n,p}$ be the free~$\F_p E_n$-module on elements called~$e_1,\ldots, e_n$ and~$e_{i,j}$ for~$i < j$. There is a short exact sequence 
\[ \begin{CD}
0 @>>> K @>>> F_{n,p} @>\psi>> M_n @>>> 0 \, , 
\end{CD}
  \]
  where~$\psi (e_i) = x_i^p$ and~$\psi (e_{i,j}) = (x_i, x_j)$. We want to show that~$K = \ker(\psi )$ is generated by the elements above. For this, we shall determine the dimension of~$M_n$ (which is easy), so that we will know the dimension of~$K$ over~$\F_p$. The work will consist in exhibiting carefully selected elements of~$K$, all obtained from the above using the~$\F_p E_n$ action, which are linearly independent over~$\F_p$ and numerous enough for us to conclude that they span~$K$.

\subsubsection{A basis for~$K$}

The dimension of $M_{n}$ is well-known.

\begin{lm} \label{lem-dim-Mn}
  With our notation: \[ \dim_{\F_p}M_{n}=1+(n-1)\cdot p^{n} \, . \] \end{lm}

\begin{proof}According to \cite{KochL}, example 6.3, we have that the minimal number of topological generators of $\Phi(\mathcal{F}_{n})$ -denoted $d(\Phi(\mathcal{F}_{n}))$- is equal to $p^{n}(n-1)+1$, therefore we can conclude by definition of the Frattini subgroup. \end{proof}

When~$\nu = (\nu_1, \ldots, \nu_n)$ is a multi-index, we set
\[X^{\nu}=X_{1}^{\nu_1}X_{2}^{\nu_2}\ldots X_{n}^{\nu_{n}} \, . \]
Note that the family $(X^{\nu})_{\nu\in\mathcal{I}}$, where $\mathcal{I}=\{0,\ldots,p-1\}^{n}$ is an $\F_p$ basis of the group algebra $\F_pE_{n}$. We will write $\mathcal{E}$ for this basis. Now we proceed to introduce distinguished elements of~$K$.

\medskip
{\em $\bullet$  The relations~$R(i,m)$ and~$R(i,j,m)$.}
\medskip

For each~$1 \le i \le n$ and $m=(m_1,\ldots, m_{n})\in\{0,\ldots,p-1\}^{n}\, , $ a multi-index such that $m$ is different from $(0,\ldots,0)\, ,$ we introduce

\[R(i,m) = \left\{ \begin{array}{l}
e_{i}\cdot X_{i}^{m_{i}}\cdot\prod\limits_{s\neq i} X_{s}^{m_{s}},\quad\text{if}\, m_{i}\neq 0\, , \\
e_{i}\cdot X_{j}^{m_{j}}\cdot\prod\limits_{\substack{s\neq i\\ s< j}} X_{s}^{m_s}+e_{(i,j)}\cdot X_{i}^{p-1}X_{j}^{m_{j}-1}\prod\limits_{s\neq i,j}X_{s}^{m_{s}},\quad\text{if}\, i<j\, , \\
e_{i}\cdot X_{j}^{m_{j}}\cdot\prod\limits_{\substack{s\neq i\\ s< j}} X_{s}^{m_s}-e_{(j,i)}\cdot X_{i}^{p-1}X_{j}^{m_{j}-1}\prod\limits_{s\neq i,j}X_{s}^{m_{s}},\quad\text{if}\, i>j\, ,
\end{array}\right.
 \]
 where $j=\max\{s|m_{s}\neq 0\}$ in the second and third case and clearly $m_{i}=0$. By virtue of the relation \eqref{PP}, we deduce that $R(i,m)$ is in the kernel of $\psi $. We have therefore found exactly \begin{equation} n\cdot (p^{n}-1) \label{C0} 
\end{equation}
vectors in the kernel so far.

By virtue of the same relation, we obtain that
 \begin{equation} (x_{i},x_{j})X_{i}^{p-1}X_{j}^{p-1}=x_{i}^{p}X_{j}^{p}=0 \, , \label{Pp2} \end{equation}
thus vectors of the form \[ R(i,j,m)=e_{(i,j)}X_{j}^{p-1}X_{i}^{p-1}\prod\limits_{k<i}X_{k}^{m_{k}}\,  \] are in the kernel, for a total amount of 
\begin{equation}
\sum\limits_{i=1}^{n}(n-i)\cdot p^{i-1} \label{C1}
\end{equation}
vectors of those form.

\medskip
{\em $\bullet$ Relations of Jacobi type.}
\medskip

Let $x_{i},x_{j},x_{k}$ be three elements of our generating system of $\mathcal{F}_{n}$ with $i\leq j\leq k$. Because of \eqref{Jacobi} and the elementary properties on the commutators, we get:
\begin{equation}
(x_{i},x_{j})X_{k}=(x_{i},x_{k})X_{j}-(x_{j},x_{k})X_{i}\, . \label{Jac}
\end{equation}

Thus elements of the form $e_{(i,j)}X_{k}-e_{(i,k)}X_{j}+e_{(j,k)}X_{i}$, where $i<j<k$, lie in the kernel. More generally, by multiplying the previous relation by $X^{m}$ verifying the following conditions
\begin{enumerate}
\item $m_{k}\neq p-1$,
\item if $i>k$ then $m_{i}=0$,
\end{enumerate}
we deduce that vectors of the following form lie also in the kernel: \[e_{(i,j)}\cdot X_{k}^{m_{k}+1}\cdot\prod\limits_{s<k} X_{s}^{m_{s}}+e_{(j,k)}\cdot X_{i}^{m_{i}+1}\cdot\prod\limits_{\substack{s\neq i\\ s\leq k}}X_{s}^{m_{s}}-e_{(i,k)}\cdot X_{j}^{m_{j}+1}\cdot\prod\limits_{\substack{s\neq j\\ s\leq k}}X_{s}^{m_{s}}\, , \]
where $m\in\{0,\ldots,p-1\}^{n}$ is a multi-index such that $m$ is different from $(0,\ldots,0)\, .$

Such vectors are denoted $jac_{1}(i,j,k,m)$; clearly $m_{k}\leq p-2$, and their number is \[(p-1)\cdot \sum_{k=0}^{n}\begin{pmatrix}
k-1 \\ 2
\end{pmatrix}\cdot p^{k-1} \, . \]

By multiplying \eqref{Jac} by $X_{k}^{p-1}$, we obtain the relation \begin{equation}
(x_{i},x_{k})X_{k}^{p-1}X_{j}=(x_{j},x_{k})X_{k}^{p-1}X_{i} \label{Jac2}
\end{equation}

Therefore, by multiplying by $X^{m}$ where $m$ verify the following conditions
\begin{enumerate}
\item $m_{j}\neq p-1$,
\item if $i>j$, $m_{i}=0$,
\end{enumerate}
 we get again vectors of the form \[e_{(i,k)}X_{k}^{p-1}X_{j}^{m_{j}+1}\prod\limits_{s<j} X_{s}^{m_{s}}-e_{(j,k)}X_{k}^{p-1}X_{i}^{m_{i}+1}\cdot\prod\limits_{\substack{s< j \\ s\neq i}} X_{s}^{m_{s}}\quad\text{where}\, i<j<k\, , \] which are in the kernel. Hence we have added in our kernel a total amount of

\[(p-1)\cdot\sum_{j=1}(j-1)(n-j)p^{j-1}\]
vectors of this form, they are denoted by $jac_{2}(i,j,k,m)$.

From now on, we set
\[\digamma=\{R(i,m),R(i,j,m),jac_{1}(i,j,k,m),jac_{2}(i,j,k,m)\} \, , \]
with the conditions on~$i,j,k,m$ given above.

\begin{lm} \label{lem-basis-ker-psi} The system $\digamma$ is a basis of $\ker\psi$. \end{lm}

\begin{proof}All vectors contained in $\digamma$ are in $\ker\psi$ by definition; we shall prove that they are linearly independent and that their number is equal to $\dim F_{n,p}-\dim M_{n}$.

{{\em Linear independence.}} Bear in mind that $\mathcal{E}$ is the basis of $F_{n,p}$ consisting of the $e_{i}\cdot X^{\nu}$ and the $e_{(i,j)}\cdot X^{\mu}$ where $\nu$ and $\mu$ are elements of $\{0,\ldots,p-1\}^{n}$. Let us define an $\F_p$-linear map $f\colon F_{n,p}\rightarrow F_{n,p}$ given on the vectors of $\mathcal{E}$ by

\[\left\lbrace\begin{array}{rcl}
f(e_{i}\cdot X^{m})&= &R(i,m) \\
f(e_{(i,j)}\cdot X_{k}^{m_{k}+1}\cdot\prod\limits_{s<k} X_{s}^{m_{s}})&=&jac_1 (i,j,k,m) \\
f(e_{(i,k)}X_{k}^{p-1}X_{j}^{m_{j}+1}\prod\limits_{s\leq i<j} X_{s}^{m_{s}})&=&jac_2 (i,j,k,m)
\end{array} \right. \]
and fixing the other vectors of $\mathcal{E}$; note that among those remaining vectors are the $R(i,j,m)$ for instance. In order to number the vectors of the basis, we will use an order relation rather than cumbersome formulae from combinatorics.

We define a total order relation on the vectors of $\mathcal{E}$ by imposing the following conditions: \begin{enumerate}
\item $e_{i}\cdot\prod X_{s}^{\nu_{s}}\leq e_{j}\cdot\prod X_{s}^{\mu_{s}}$ if and only if $i<j$ or  $i=j$ and either $|\nu|<|\mu|$ or if $|\nu|=|\mu|$ then we use the lexicographic order.
\item $e_{(i,j)}\cdot\prod X_{s}^{\nu_{s}}\leq e_{(k,l)}\cdot\prod X_{s}^{\mu_{s}}$ if and only if one of the following condition is true
\begin{enumerate}
\item $i<k$
\item if $i=k$ then one of the following must be true:
\begin{enumerate}
\item $j<l$,
\item $|\nu|<|\mu|$,
\item $\nu\preceq\mu$ where $\preceq$ is the lexicographic order
\end{enumerate}
\end{enumerate}
\item $e_{i}\cdot\prod X_{s}^{\nu_{s}}\leq e_{(j,k)}\cdot\prod X_{s}^{\mu_{s}}$.
\end{enumerate} \

The matrix associated to~$f$ in the canonical basis, thus ordered, is lower triangular with $1$'s on the diagonal, as is readily checked (when defining the elements of~$\digamma$, we have always given the formulae so that the leftmost term is the lowest for the order relation).

So~$f$ is invertible, and the image of the canonical basis under~$f$ is another basis for~$F_{n,p}$. This proves in particular that the elements of~$\digamma$ are linearly independent.

{{\em Cardinality.}} By using the formula previously given, we can get: $\dim_{\F_p}\ker\psi={n \choose 2} p^{n}-1.$

However \begin{equation}
\begin{array}{rcl}
(p-1)\cdot\sum\limits_{k=0}^{n-1}\begin{pmatrix}
k \\ 2
\end{pmatrix} p^{k} &=& 
 \begin{pmatrix} n-1 \\ 2 \end{pmatrix} p^{n} -\sum\limits_{k=1}^{n-1}(k-1)\cdot p^{k} \, , \\

\end{array}
\label{C2}
\end{equation}
in the same fashion
\begin{equation}\begin{array}{rcl}
(p-1)\cdot\sum\limits_{j=1}^{n}(n-j)(j-1)p^{j-1}
&=&\sum\limits_{j=1}^{n-1}(2j-n)p^{j} \, , \\

\end{array}
\label{C3}
 \end{equation}
by adding the previous equalities we get: 
\[
\begin{array}{rcl}
\eqref{C1}+\eqref{C2}+\eqref{C3}&=&\begin{pmatrix}n-1 \\ 2 \end{pmatrix}p^{n}+\sum\limits_{k=0}^{n-1}kp^{k}-\sum\limits_{k=0}^{n-1}kp^{k-1}+n\sum\limits_{i=1}^{n-1}p^{i-1}-n\sum\limits_{i=1}^{n-1}p^{i}+\sum\limits_{k=0}^{n-1}p^{k} \\
&=&\begin{pmatrix}
n-1 \\ 2
\end{pmatrix}p^{n}+\sum\limits_{k=0}^{n-1}kp^{k}-\sum\limits_{k=0}^{n-2}(k+1)p^{k}+n(1-p)\sum\limits_{k=0}^{n-2} p^{k}+\sum\limits_{k=0}^{n-1}p^{k} \\
&=&\begin{pmatrix}
n-1 \\ 2
\end{pmatrix}p^{n}+n(1-p^{n-1})+\sum\limits_{k=1}^{n-1}p^{k}-\sum\limits_{k=0}^{n-2}p^{k} \\
&=&\begin{pmatrix}
n-1 \\ 2
\end{pmatrix}p^{n}+n-1 \, .
\end{array} \]

If we add this to \eqref{C0}, we obtain the desired cardinality. 
\end{proof}

From this lemma we can deduce the following three corollaries.

 \begin{cor}\label{Base}The system formed by the vectors $(x_{i}^{p})_{i\in\{1,\ldots,n\}}$ and the vectors $(x_{i},x_{j})X^{\nu}$ such that $\nu$ verifies the following conditions \begin{enumerate}
\item $\max\{s|\nu_{s}\neq 0\}\leq j$
\item $\nu_{i}\neq p-1 $ or $\nu_{j}\neq p-1$
\item if $\nu_{j}=p-1$, then $\nu_{k}=0$ for $k\in\{i+1,\ldots,j-1\}$.
\end{enumerate}
forms a basis of $M_{n}$. \end{cor}
 
 \begin{proof} Let 
\[ \mathcal{B} = f(\mathcal{E}) \smallsetminus \digamma \, ,   \]
where~$\mathcal{E}$ is our usual basis for~$F_{n,p}$ and~$f$ is the endomorphism defined in the proof of the lemma. Then~$\psi (\mathcal{B})$ is a basis for~$M_n$. However, a vector of $v \in \mathcal{E}$ is in $f^{-1}(\digamma)$ if and only if one of the following condition is true:

\begin{enumerate}
\item if $v=e_{i}X^{\nu}$ where $\nu\neq(0,\ldots,0)$.
\item if $v=e_{(i,j)}X^{p-1}_{i}X^{p-1}_{j}\prod\limits_{k\notin\{i,j \}}X_{s}^{m_{s}}\, .$
\item if $v=e_{(i,j)}X_{k}^{m_{k}}\prod\limits_{s<k}X^{m_{s}}_{s}\, ,$ where $m_{k}\neq 0\, .$
\item if $v=e_{(i,k)}X_{k}^{p-1}X_{j}^{m_{j}}\prod\limits_{s\leq i<j} X_{s}^{m_{s}} \, ,$ where $m_{j}\neq 0 \, .$
\end{enumerate}

Negating this conditions, and keeping in mind that~$f(v)=v$ if~$v$ is not in~$f^{-1}(\digamma)$, we obtain the announced result.
\end{proof}

\begin{cor}\label{Pres}The module $M_{n}$ admits the following presentation by generators and relations: \begin{itemize}
\item its generators are the $x_{i}^{p}$ and the $(x_{i},x_{j})$ where $i$ and $j$ are in $\{1,\ldots,n\}$ and $i<j$.
\item The relations are given by \begin{enumerate}
\item $x_{i}^{p}\cdot X_{i}=0\, ,$
\item $x_{i}^{p}\cdot X_{j}=(x_{j},x_{i})X_{i}^{p-1}$ if $i>j$,
\item $x_{i}^{p}\cdot X_{j}=-(x_{i},x_{j})X_{i}^{p-1}$ if $i<j$,
\item $(x_{i},x_{j}) X_{k}+(x_{j},x_{k}) X_{i}-(x_{i},x_{k})X_{j}=0\, ,$ where $i<j<k$.
\end{enumerate}
\end{itemize} \end{cor}

Notice that, alternatively, we could have used generators~$(x_i, x_j)$ for~$i \ne j$ (rather than just~$i<j$), add the relation~$(x_i, x_j) = - (x_j,x_i)$, and then delete relation (3) which is now redundant with (2). Also (4) can then be re-written in a more symmetrical form.

\begin{proof}Let $R_{n}$ be the module defined by the presentation of the corollary. It should be remarked that there exists an obvious map of modules from $R_{n}$ onto $M_{n}$, for the relations verified in $R_{n}$ are verified in $M_{n}$ too: therefore it is clear that
\[\dim_{\F_p}M_{n}\leq \dim_{\F_p} R_{n} \, . \]
 By looking closer to the proof of the previous corollary, we see that we only used the relations mentioned in the corollary in order to construct $\digamma$, therefore we can show exactly by re-writing the proof of the corollary that
 \[\dim_{\F_p}R_{n}\leq\dim_{\F_p}M_{n} \, . \]
 So the dimensions are equal, and the obvious epimorphism is an isomorphism.  \end{proof}

\begin{rem} Thanks to this presentation, it would have been easy to prove the isomorphism
\[\Phi(\mathcal{F}_{n})/\Phi^{(2)}(\mathcal{F}_{n})\simeq\Omega^{2}(\F_p)\, , \]
without using Proposition \ref{GB}. Indeed, we could have computed a presentation by generators and relations of the kernel of the map
\[(\F_{p}E_{n})^{n}=\langle e_{1},\ldots,e_{n} \rangle\xrightarrow{e_{i}\mapsto X_{i}}I(G)\, , \]and we would have remarked that the this presentation coincides with the one we found for $\Phi(\mathcal{F}_{n})/\Phi^{(2)}(\mathcal{F}_{n})$. \end{rem}

\subsection{Link with the cohomology}
As seen earlier, the map $\kappa\colon\Omega^{-1}(\F_p)\longrightarrow\omega_2 (\F_p)$ contains some precious piece of information we need in order to describe the module $J(\K)$: here, knowing completely the module structure, we describe the map $\kappa$. It should be recalled that
 \[H^{\bullet}(E_n,\F_p)\simeq \F_p [\zeta_1,\ldots,\zeta_{n}]\otimes\Lambda[\eta_1,\ldots\eta_n]\, , \]where the generators $\zeta_i$ have weight 2, whereas the generators $\eta_i$ have weight 1. (see\cite{CoRi}) Remember that the ideal of augmentation $I(E_n)$ is stably isomorphic to $\Omega(\F_p)$. We aim here to make explicit the following isomorphisms

 \[\left\lbrace\begin{array}{rcll}
   H^1 (E_n ,\F_p) &\stackrel{\simeq}{\longrightarrow } &\sthom (I(E_n),\F_p) & ~\textnormal{written}~ x\mapsto \baro{x} \, , \\
   H^1 (E_n ,\F_p) & \stackrel{\simeq}{\longrightarrow } &  \sthom (M_n, I(E_n)) & ~\textnormal{written}~ x\mapsto \baru{x} \, ,\\
    H^2 (E_n ,\F_p) & \stackrel{\simeq}{\longrightarrow } &  \sthom (M_n,\F_p) & ~\textnormal{written}~ x\mapsto \ove{x}\, . 

 \end{array}\right.  \]

 Now, note that we shall work not exactly with morphisms in the stable-module category, but with their lifts which are truly morphisms of modules. Furthermore a mere bijection of set is not sufficient, if we want to keep track of the ring structure: it has to be compatible with the cup product. To be clear, we would like the following equality to hold:

\[\ove{\eta_{i} \smile \eta_{j}}=\baro{\eta_{i}}\circ\baru{\eta_{j}} \, . \]

Let us set the stage by the first isomorphism. Since $I(E_n)$ is generated as a module by the elements denoted $X_i =x_i -1$, it is sufficient to define a morphism on them.
\[\fonction{\baro{\eta_{i}}}{I(G)}{\F_p}{X_j}{\left\lbrace\begin{array}{lcr}
0 & \text{if}& i\neq j \\
1 & \text{if}& i=j
\end{array}\right.} \]
It is obvious that such maps extend well to morphisms of modules and that they generate $\hom_{\F_p E_n}(I(G),\F_p)$. Now, we consider the isomorphism
\[H^1 (E_n ,\F_p)\longrightarrow \sthom(I(G),\F_p)\simeq \operatorname{Span}(\baro{\eta_1},\ldots,\baro{\eta_n}) \]
which maps~$\eta_i$ to~$\baro{\eta_i}$, extended by linearity.

We turn to the second isomorphism. Because $M_n$ is generated as a module by $(x_{i}^{p})_{1\leq i\leq n}$ and $((x_i,x_j))_{1\leq i<j\leq n}$ it is sufficient to define maps on the set $\mathcal{S}$ of those elements. We can therefore put

\[\fonction{\baru{\eta_{i}}}{M_n}{I(E_n)}{x\in\mathcal{S}}{\left\lbrace\begin{array}{lcr}
X_j & \text{if}& x=(x_{i},x_{j}) \\
-X_{j} & \text{if} &x=(x_{j},x_{i}) \\
0 & & \text{otherwise}
\end{array}\right.}\, . \]

It is less obvious that this expression extends well to a map of modules, we let the reader verify this fact by a simple computation. It remains to prove that those elements are a basis of $\sthom (M_n , I(E_n) )$.

\begin{lm}The equivalence classes of the family of maps $(\baru{\eta_{i}})_{1\leq i\leq n}$ form a basis of $\sthom (M_n, I(E_n))$. \end{lm}

\begin{proof}
It is easy to see that this family is free in $\hom (M_n ,I(E_n))$ but we have in fact to check that
\[\psi=\sum a_{i}\baru{\eta_{i}} \]factors through a projective module if and only if $(a_1,\ldots,a_n)=(0,\ldots,0)$. So let us suppose the converse: assume there exists a non-zero linear combination and a projective module $P$ such that the following diagram commutes.
\[\xymatrix{M_n \ar[rr]^{\sum a_i \baru{\eta_{i}}} \ar[rd]^{f} & & I(E_n) \\
& P \ar[ru]^{g} &} \]
Now it is clear that since for all $x\in M_n$, $x\cdot\No=0$, then $\Im f\subset \Rad(P)$, so that $\Im g\circ f\subset\Rad(I(E_n))$. Since $(a_1,\ldots,a_n)\neq (0,\ldots,0)$ there exists $j$ such that $a_j \neq 0$. Let us then apply $\psi$ to $(x_{j},x_{j+1})$ (or if $j=n$ to $(x_{n-1},x_{n})$):
\[\psi((x_j , x_{j+1}))=\sum a_i \baru{\eta_{i}}(x_{j},x_{j+1})=a_{j}X_{j}-a_{j+1}X_{j+1} \, , \]
however $a_{j}X_{j}-a_{j+1}X_{j+1}\neq 0$ and is not in $\Rad(I(E_n))$, which is absurd. 
\end{proof}

It should be added that

\[\baro{\eta_{i}}\circ\baru{\eta_{i}}\colon M_n \longrightarrow \F_p  \]
is zero. Therefore the isomorphism (implicitly) given between $H^1 (E_n,\F_p)$ and $\sthom (M_n ,I(E_n))$ is coherent with the one given between $\sthom(I(E_n),\F_p)$.

Now, we can turn to the isomorphism between $\sthom( M_n ,\F_p)$ and $H^2(E_n,\F_p)$. Let us set

\[\fonction{\ove{\zeta_{i}}}{M_n}{\F_p}{x\in\mathcal{S}}{\left\lbrace\begin{array}{lcr}
1 & \text{if}& x=x_{i}^p \\
0 & & \text{otherwise}
\end{array}\right.}\, . \]
The maps in bijection with the cup products are simply given by composing the two sets of representatives we have previously given, hence:
\[\fonction{\ove{\eta_{i}\smile\eta_j}}{M_n}{\F_p}{x\in\mathcal{S}}{\left\lbrace\begin{array}{lcr}
1 & \text{if}& x=(x_{i},x_{j}) \\
-1 &\text{if}& x=(x_{j},x_{i}) \\
0 & & \text{otherwise}
  \end{array}\right.}\, . \]
Note that those are again well defined morphisms of modules and a basis of $\hom (M_n,\F_p)$ which is again a basis of $\sthom (M_n,\F_p)$, since its cardinal is equal to the dimension of $H^2 (E_n,\F_p)$.

Let us come back to the short exact sequence \ref{SE}. In our situation, the monomorphism from $\Omega^{-1}(\F_p)\simeq\langle r|r\cdot\No=0 \rangle$ into $\Omega^{2}(\F_p)\simeq M_n$ is given by
\[\kappa
\colon r\mapsto\left\lbrace\begin{array}{rcl} (x_1,x_2)+(x_3,x_4)+\ldots (x_{n-1},x_n)  &\text{if} & \xi_{p^2}\in\k \\
x_{1}^{p}+(x_1,x_2)+(x_3,x_4)+\ldots (x_{n-1},x_n)  &\text{if} & \xi_{p^2}\notin\k
\end{array}\right.\, , \]
indeed the generator of $\Omega^{-1}(\F_p)$ is sent into the equivalence class of $[\Delta]$ modulo $\Phi^{(2)}(\mathcal{F}_{n}$, where $\Delta$ is the \dem relation:
\[\Delta=x_{1}^{p^{k}}(x_1,x_2)(x_3,x_4)\ldots(x_{n-1},x_{n})\in\Phi(\mathcal{F}_{n})\, . \]

However, such map is also a (Tate)-cohomology class of degree $-3$, and, according to Tate duality, it induces a linear form on $H^2 (E_n,\F_p)$, let us call it $\kappa^{\#}$. 

The action of $\kappa^\#$ will be soon described and can be linked to the Hasse invariant of $\k$. The inflation map gives birth to an isomorphism between the groups $H^{1}(\mathcal{D}_{k,n},\F_p)$ and $H^1 (\mathcal{F}_{n},\F_p)\simeq H^1 (E_{n},\F_p)$ which is obviously denoted by $\operatorname{inf}$. Let $m$ denote the map from $H^{1}(E_n,\F_p)\otimes H^{1}(E_n,\F_p)$ to $H^{2}(E_n,\F_p)$ induced by the cup-product and in a similar fashion, let $\tilde{m}$ be the one given by the cup product of $H^{\bullet} (\mathcal{D}_{k,n},\F_p)$. It has to be remembered that $H^2 (\mathcal{D}_{k,n},\F_p)$ is isomorphic to the Brauer group (modulo $p$) of $\k$, denoted here $\operatorname{Br}_{p}(\k)$ and the Hasse invariant $\operatorname{Inv}$ gives an isomorphism between $\operatorname{Br}_{p}(\k)$ and $\F_p$ (see \cite[\S II.3, II.4]{Gui} for a complete description of these notions).

In a certain way, the map $\kappa^{\#}$ describes the Hasse invariant on the decomposable part of the cohomology:

\begin{prop}The action of $\kappa^{\#}$ on the generators of $H^2 (E_n, \F_p)$ is as follows:
\[\begin{array}{rcl}
\kappa^{\#}(\eta_{i}\smile\eta_{j})&=&\left\lbrace\begin{array}{rcll}
1 & \text{if} & i=2k+1,\,\, j=i+1& (k\in\{0,\ldots\floor{\frac{n}{2}}\})\\
0 & \text{otherwise}& &
\end{array}\right. \, ,\\
\kappa^{\#}(\zeta_{i})&=&\left\lbrace\begin{array}{rcl}
1 & \text{if} & i=1 \text{ and } \xi_{p^2}\notin\k \\
0 & \text{otherwise}&
\end{array}\right. \, .
\end{array} \] Moreover, the following diagram is commutative
\[\xymatrix{ H^1 (E_n ,\F_p)\otimes H^1 (E_n ,\F_p) \ar[r]^{m} \ar[dd]^{\operatorname{inf}^{\otimes2}} & H^2 (E_n,\F_p) \ar[rd]^{\kappa^\#} & \\
 & & \F_p \\ H^1 (\mathcal{D}_{k,n},\F_p)\otimes H^1 (\mathcal{D}_{k,n},\F_p) \ar[r]^{\tilde{m}} & H^2 (\mathcal{D}_{k,n},\F_p)\simeq \operatorname{Br}_{p}(\k)\ar[ru]^{\operatorname{Inv}}  
 }\, . \]\end{prop}

\begin{proof} To prove this statement, it is sufficient to check that stably speaking the maps \[\begin{array}{rcccl}
\ove{\eta_{i}\smile\eta_{j}}\circ\kappa&\colon&\Omega^{-1}(\F_p)&\longrightarrow&\F_p\\
\ove{\zeta_{s}}\circ\kappa&\colon&\Omega^{-1}(\F_p)&\longrightarrow&\F_p
\end{array}\] send the generator $r$ of $\Omega^{-1}(\F_p)$ to $1$ when $i=2k+1$ and $j=i+1$ or when $s=1$ and $\xi_{p^2}\in\k$, and otherwise vanish. According to the previous expressions, it is quite clear.

Furthermore the isomorphism between $H^2 (\mathcal{D}_{n,k},\F_p)$ and $\F_p$ is given by the Hasse invariant. According to \cite[Proposition 4]{Lab}, which describes the cup products, the previous diagram commutes.

\end{proof}

\subsection{From $M_n$ to $J^*$} 
 
Now that we have a proper presentation by generators of $M_{n}$, we can find one of $J$ without any difficulty. In order to clarify our results we assume that $\xi_{p^2}\in\k$ if it is not the case, only small changes have to be made, which will be pointed out along the text. As pointed out in the previous subsection, the generators of $M_n$ are in bijection with a basis of $H^2(E_n,\F_p)$, in fact the relations are indexed by a basis of $H^3 (E_n,\F_p)$. Let us label them in the following way
\[\begin{array}{rcll}
\zeta_{i}\smile\eta_{i}&\colon &x_{i}^p X_i =0 & \\
\zeta_{i}\smile\eta_{j}&\colon &x_{i}^p X_j +(x_i,x_j)X_{i}^{p-1} =0 &i<j \\
\zeta_{i}\smile\eta_{j}&\colon &x_{i}^p X_j -(x_j,x_i)X_{i}^{p-1} =0 &i>j \\
\eta_{i}\smile\eta_{j}\smile\eta_{k}&\colon &(x_i ,x_j)X_k +(x_j,x_k)X_i -(x_i,x_k)X_j =0 &1\leq i<j<k\leq n 
\end{array} \]

We can state the following lemma:

\begin{lm} \label{lem-pres-J}
  The module~$J$ can be presented as:
  \begin{align*}
    J & \cong \omega_2(\F_p) / (\Delta ) \\
      & \cong \langle x_{i}^p, (x_i , x_j) ~|~ \zeta_{i}\smile\eta_{j},\eta_{i_0}\smile\eta_{i_1}\smile\eta_{i_2}, \Delta  \rangle
\end{align*}
  where~$1 \le i < j \le n$, and $1\leq i_0 <i_1 <i_2\leq n$; as for~$\Delta $, it stands for the relation 
\[ (x_1,x_2)+(x_3,x_4)+\ldots+(x_{n-1},x_n)=0 \, .   \]
\end{lm}

\begin{proof}
  The Demu\v{s}kin group $G_\k(p) \cong \mathcal{D}_{k,n}$ is the quotient of the free pro-$p$-group $\mathcal{F}_{n}$ by the relation
  \[\Delta\colon x_{1}^{p^{k}}(x_1,x_2)(x_3,x_4)\ldots(x_{n-1},x_{n})=1 \, , \]
  and~$k \ge 2$ from our assumption that~$\xi_{p^2} \in \k$ (see \S\ref{subsec-demushkin}). It is clear that~$\Delta \in \Phi (\mathcal{F}_n)$.

When~$G$ is a finitely generated pro-$p$ group, and~$K$ is a closed subgroup, one sees easily that~$\Phi (G)$ maps onto $\Phi (G/K)$ under the quotient map $G \longrightarrow G/K$. Moreover, if~$K \subset \Phi (G)$, then~$\Phi (G/K)$ can be identified with~$\Phi (G)/K$, clearly. If $N$ denotes the smallest closed, normal subgroup containing $\Delta$, we have~$N \subset \Phi (\mathcal{F}_n)$ and so there is an exact sequence 
\[ \begin{CD}
0 @>>> N @>>> \Phi (\mathcal{F}_n) @>>> \Phi (\DD_{k,n}) @>>> 1 \, . 
\end{CD}
  \]
  Now by the same reasoning, we see that~$\Phi ^{(2)}(\FF_n)$ maps onto $\Phi ^{(2)} (\DD_{k,n})$; it follows easily that there is another exact sequence 
\[ \begin{CD}
0 @>>> N / (N \cap \Phi ^{(2)}(\FF_n)) @>>> \Phi (\mathcal{F}_n) / \Phi ^{(2)}(\FF_n) @>>> \Phi (\DD_{k,n}) / \Phi ^{(2)}(\DD_{k,n}) @>>> 1 \, . 
\end{CD}  \]
This says in other notation, using our identification of~$M_n$ with~$\omega_2(\F_p)$, that the kernel of~$\omega_2(\F_2) \longrightarrow J$ is generated, as~$\F_p E_n$-module, by~$\Delta $. 
\end{proof}

\begin{rem}
If $\k$ does not contain $\xi_{p^2}$, then the relation $\Delta$ becomes
\[x_{1}^{p}+(x_1,x_2)+(x_3,x_4)+\ldots+(x_{n-1},x_n)=0\, . \]
Furthermore if $p=2$ and $n$ is odd, the relation is
\[ x_{1}^{2}+(x_2,x_3)+(x_4,x_5)+\ldots +(x_{2s},x_{2s+1}) \]
\end{rem}

We can even go a little deeper into the computations: to achieve our goal of studying some invariants, we would like in fact to find a basis of $J^*$. Bear in mind that the projection from $\mathcal{F}_n$ onto $\mathcal{D}_{k,n}$ induces an epimorphism $\pi_{\Phi(\mathcal{F}_n)}\colon M_n\longrightarrow J^*$, hence the image of the previously found basis $\mathcal{B}$ of $M_n$ is a generating system of $J^*$. Nevertheless we have to get rid of some vectors in order to have a proper basis: those vectors can be found through the study of the kernel of $\pi_{\Phi(\mathcal{F}_n)}$ which is generated by $\Delta$. In order to do so, we need to introduce the two following family of maps on multi-indices
\[\fonction{\delta_{i}}{\Z^n}{\Z^n}{(\nu_1,\ldots,\nu_n)}{(\nu_1,\ldots,\nu_{i}-1,\ldots,\nu_{n})} \]
and
\[\fonction{\gamma_{i}}{\Z^n}{\Z^n}{(\nu_1,\ldots,\nu_n)}{(\nu_1,\ldots,\nu_{i}+1,\ldots,\nu_{n})}  \]

Now we can state a technical lemma -the reader can skip its proof, for it is a silly and tedious computation.

\begin{lm}For any multi-index $\nu\in\{ 0,\ldots,p-1\}^{n}$ we set
\[I(\nu)=\{k|\nu_{2k-1}\neq p-1\,\text{or}\, \, \nu_{2k}\neq p-1 \} \, , \]
and
\[\mu(\nu)=\max\{k|\nu_{k}\neq 0 \}\, .\]
Then the following equality holds in $M_n$ \[\Delta\cdot X^{\nu}=\sum\limits_{\substack{s=1\\s\in I(\nu)}}^{\lceil\frac{\mu(\nu)}{2} \rceil-1}(x_{2s-1},x_{\mu(\nu)})\cdot X^{\gamma_{2s}\circ\delta_{\mu(\nu)}(\nu)}-(x_{2s},x_{\mu(\nu)})\cdot X^{\gamma_{2s-1}\circ\delta_{\mu(\nu)}(\nu)}+\sum\limits_{\substack{s=\lceil \frac{\mu(\nu)}{2}\rceil\\s\in I(\nu)}}^{\frac{n}{2}}(x_{2s-1},x_{2s})\cdot X^{\nu}\, . \] \end{lm}

\begin{proof}
First note that
\[\Delta\cdot X^{\nu}=\sum\limits_{k\in I(\nu)}(x_{2k-1},x_{2k})X^{\nu}+\sum\limits_{k\notin I(\nu)}(x_{2k-1},x_{2k})X^{\nu}\, , \]
and since $(x_i,x_j)X_{i}^{p-1}X_{j}^{p-1}=0$ according to \ref{Pp2}, we have the following equality
\[(x_{2k-1},x_{2k})X^{\nu}=0\, , \]

if $k\notin I(\nu)$. Therefore, it is sufficient to sum upon the integers which are in $I(\nu)$. Now, let us split the sum in two:
\[\sum\limits_{k\in I(\nu)}(x_{2k-1},x_{2k})X^{\nu}=\sum\limits_{\substack{k\in I(\nu)\\k<\lceil\frac{\mu(\nu)}{2}\rceil}}(x_{2k-1},x_{2k})X^{\nu}+\sum\limits_{\substack{k\in I(\nu)\\k\geq\lceil\frac{\mu(\nu)}{2}\rceil}}(x_{2k-1},x_{2k})X^{\nu} \, . \]
Remark that the vectors in the second sum are all in the basis $\mathcal{B}$, whereas those on the first sum are not. To address this issue, we will use in fact the Jacobi relation (the one denoted $\eta_{2k-1}\smile\eta_{2k}\smile\eta_{\mu(\nu)}$), hence
\[\begin{array}{rcl}
(x_{2k-1},x_{2k})X^{\nu}&=&(x_{2k-1},x_{2k})X_{\mu(\nu)}X^{\delta_{\mu(\nu)}(\nu)}\\
&=& (x_{2k-1},x_{\mu(\nu)})X_{2k}\cdot X^{\delta_{\mu(\nu)}(\nu)}-(x_{2s},x_{\mu(\nu)})X_{2s-1}\cdot X^{\delta_{\mu(\nu)}(\nu)} \, .
\end{array} \]
All those vectors -when they are non zero- are in the basis $\mathcal{B}$, and the proof is complete.
\end{proof}

\begin{rem}
If $\xi_{p^2}$ is not in $\k$, then it is necessary add a term $x_{1}^{p}\cdot X^{\nu}$. If it is non zero it is equal to $(x_{1},x_{\mu(\nu)})X_{1}^{p-1}X^{\delta_{\mu(\nu)}(\nu)}$.
\end{rem}

Having found a basis for $\ker\pi_{\Phi(\mathcal{F}_n)}$, we can now find a basis for  $J^*$: the following basis remains exactly the same no matter if $\k$ does or does not contain $\xi_{p^2}$.

\begin{prop}Let $\mathcal{B}$ the basis obtained for $\omega_2 (\F_p))$. A basis $\mathcal{B}'$ of $J^*$ is given by the image of $\mathcal{B}\backslash V$, where $V$ is the set consisting of the vectors $(x_i ,x_j)X^{\nu}$ verifying the following conditions:
 \begin{enumerate}
\item If $i=n-1$ there is no condition.
\item If $i\neq n-1$, the following conditions have to be verified:
\begin{enumerate}
\item $\nu_{n}=p-2$
\item if $i$ is even
\begin{enumerate}
\item $\nu_{i-1}\neq 0 \, .$
\item $\forall j\in\{n-1, n-2,\ldots,i\},\quad \nu_{j}=p-1 \, .$
\end{enumerate}
\item if $i$ is odd
\begin{enumerate}
\item $\nu_{i+1}\neq 0 \, .$
\item $\forall j\in\{n-1,\ldots,i+1\},\quad \nu_{j}=p-1\, .$
\end{enumerate}
\end{enumerate}
\end{enumerate} \label{Basis} \end{prop}

\begin{proof}The main idea of the proof is simply get rid of a single vector for each relation $\Delta\cdot X^{\nu}$. The given vectors in $V$ may be found in the formula of the previous lemma.

If $i=n-1$, note that, since the vector $(x_{i},x_{j})X^{\nu}$, $j=n-1$, hence it appears in the sum of the previous lemma as the term most to the right. Now, suppose that $i\neq n-1$ and the other conditions are of course verified. The vector we wish to eliminate appears in the relation $\Delta\cdot X^{\gamma_{n}(\nu)}$: more precisely it appears in the last term of the first sum.

As long as $\nu\neq (p-1,\ldots,p-1)$, $\Delta\cdot X^{\nu}$ is not zero and formally speaking taking again our order relation, we can verify that this is truly a free family of $M_{n}$; therefore we obtain that we are allowed to get rid of those vectors. Thus the system $\pi_{\Phi(\mathcal{F}_n)} (\mathcal{B}\backslash V)$ is still a generating system and has the expected cardinality, hence the proposition.\end{proof}

\subsection{Computing some invariants}

In their article (\cite{MKEG}), the  authors introduced some invariants for various fields, including local fields and $C$-fields; as we focus only on local fields, we do not really aim to work in such a general frame as they do, but we will still show how their results, obtained for $p=2$, might be extended when $p$ is an odd prime number and for a local field. Note that the following results mainly depend upon Proposition \ref{Basis} which holds in every case, provided that $p\neq 2$: therefore what follows is true if $\xi_{p^2}\notin\k$. Nevertheless we will assume that $\xi_{p^2}\in\k$ in order to make our computations less cumbersome.

Remember that if $M$ is an $R$-module, the socle series of $M$ is defined by

\[\left\lbrace\begin{array}{rcl}
\Soc^0 (M) &=& \{0\} \\
\Soc^n (M)/\Soc^{n-1}(M)&=&\Soc(M/\Soc^{n-1}(M))
\end{array}\right. \]

Therefore we have

\[\Soc^0 (M) \subset \Soc^1 (M)\subset\ldots \, , \]
if $M$ of finite type, we are sure that there exists a minimal integer $l(M)$ such that $\Soc^{l(M)}=M$ such integer is called the length of $M$; we will rather study -as long as it is possible- the radical series. Bear in mind that the length of the radical series is equal to $l(M)$ (see \cite{BensCR}). In \cite[Theorem 5.2,5.3,5.15]{MKEG}, a formula was proved, which related the length of $\Phi(\mathcal{G}_{\k}(2))/\Phi^{2}(\mathcal{G}_{\k}(2))$ to the $2$-cohomological dimension $\cd_{2}(\mathcal{G}_{\k}(2))$ of $\mathcal{G}_{\k}(2)$. Indeed, for every $C$-field, which is not formally real, the following equality holds:

\[l(\Phi(\mathcal{G}_{\k}(2))/\Phi^{2}(\mathcal{G}_{\k}(2)))+\cd_{2}(\mathcal{G}_{\k}(2))=n+1 \, . \]

Since it is not our purpose here to study such a general class of fields, we shall not give further details about them: the only thing that shall be known is that local fields are $C$-fields. Now, we can state the following proposition, which properly extends the previous formula, in our context.

\begin{prop}If $\mathbf{k}$ is a local field, then the following identity holds
\[l(\Phi(\mathcal{G}_{\k}(p))/\Phi^{2}(\mathcal{G}_{\k}(p)))+\cdp(\mathcal{G}_{\k}(p))=(p-1)n+1 \, . \] \end{prop}

\begin{proof}Remember that there are only two possibilities for $\mathcal{G}_{\k}(p)$: the free pro-$p$-group $\mathcal{F}_n$ and a \dem group $\mathcal{D}_{k,n}$. Let us distinguish between those cases.

First, let us assume that $\mathcal{G}_{\k}(p)$ is a free pro-$p$-group, then $\cdp(\mathcal{G}_{\k}(p))=1$ (see \cite[\S 3, Proposition 16]{CoGal}). Now, note that the length of $M_n$ is in fact equal to $(p-1)n$: indeed $\Rad^{n(p-1)}(\F_p E_{n})M_{n}=\{0\}$, because if $\nu=(p-1,\ldots,p-1)$, then $x_{i}^{p}X^{\nu}=0$ and $(x_i,x_j)X^{\nu}=0$, for $(x_i ,x_j)X_{i}^{p-1}X_{j}^{p-1}=0$ furthermore $\Rad^{n(p-1)-1}(\F_p E_{n})M_{n}\neq\{0\}$, for $(x_1,x_{n})X^{p-2}_{n}\prod\limits_{i<n-1}X_{i}^{p-1}$ is non-zero: it is a vector of the basis previously found.

Secondly, let us suppose that $\mathcal{G}_{\k}(p)$ is a \dem group. By definition $\cdp(\mathcal{G}_{\k}(p))=2$, furthermore we claim that the length of $J^*$ is equal to

\[l(J^*)=(p-1)n-1 \, . \]
Before doing our computations, let us set the map $\sigma\in\mathfrak{S}_{n}$ defined by
\[\sigma=(1,2)(3,4)\ldots(n-1,n)\, , \]
note that the permutation $\sigma$ is meant to send an integer $i$ to the integer $j$ such that $(x_{i},x_{j})$ or $(x_{j},x_{i})$ appears in the \dem relation. 
Now, let us compute the length of $J^*$. We remark that $\Rad^{(p-1)n-1}(\F_p E_n)\cdot J^*=\{0\}$, indeed we have
\[\begin{array}{rclr}
(x_{i},x_{n})X_{n}^{p-2}\prod\limits_{j\leq n}X_{j}^{p-1}&=&(x_i , x_n)X_{\sigma(i)}X_{\sigma(i)}^{p-2} X_{n}^{p-2}\prod\limits_{j\notin\{n,i\}}X_{j}^{p-1}&\\
&=&(x_{i},x_{\sigma(i)})X_{\sigma(i)}^{p-2}\prod\limits_{j\neq f(i)}X_{j}^{p-1}&(\eta_{i}\smile\eta_{\sigma(i)}\smile\eta_{n}) \\
&=&-\sum\limits_{\substack{s=0 \\ s\neq i}}^{\frac{n}{2}} (x_{2s-1},x_{2s})X_{\sigma(i)}^{p-2}\prod\limits_{j\neq \sigma(i)}X_{j}^{p-1} & (\Delta) \\
&=&0&\, . \end{array}
\]
Yet $\Rad^{(p-1)n-1}(\F_p E_n)J^* \neq \{0\}$, since for instance $(x_1,x_n)X_{n}^{p-2}X_{2}^{p-2}\prod\limits_{\substack{j=1 \\ j\neq 2}}^{n-1}X_{j}^{p-1}$ is in the basis  previously given.

This finishes the proof.
\end{proof}

Another noteworthy statement of this paper which is only proved when $p=2$, but which could be proved \textit{mutatis mutandis}, is the following proposition.

\begin{prop}[\cite{MKEG} 3.10]In the mod $p$ Lyndon-Hochschild-Serre spectral sequence for the group extension 
\[\xymatrix{1 \ar[r] & \Phi(\mathcal{G}_{\k}(p)) \ar[r] & \mathcal{G}_{k}(p) \ar[r] & E_n \ar[r] & 1}\, , \]
we have $E^{1,1}_{\infty}\simeq\Soc^2 (J)/\Soc(J)$.
  \end{prop} 

After a concrete computation of the dimension of $E^{1,1}_{\infty}$, the authors check it is equal to the dimension of $\Soc^2 (J)/\Soc(J)$, when $p=2$ (\cite[Example 5.6]{MKEG}). In this case they obtained:
\[\left\lbrace\begin{array}{rcl}
\dim_{\F_{p}}\Soc^{2}((\Omega^{2}(\F)^*)/\Soc^{1}(\Omega^{2}(\F))^{*})&=&\frac{n(n+1)(n+4)}{6} \\
\dim_{\F_{p}} \Soc^{2}(J) / \Soc(J) & = &\frac{n(n-2)(n+2)}{3}
\end{array}\right. \]
In fact when $p$ is even, a peculiar phenomenon occurs, as a consequence of the relation
\[x_{i}^2 \cdot X_{j}=(x_{i},x_j)X_{i}^{2-1=1} \, . \]
Without this fact, the formulae are slightly different from the previous one because we have to take account of terms $x_{i}^{p}X_{j}$.
\begin{prop}With our notations we have in one hand
\[\dim_{\F_{p}}\Soc^{2}(\Omega^{2}(\F)^*)/\Soc^{1}(\Omega^{2}(\F)^{*})=\frac{n(n-1)(n+4)}{3} \, , \]
and in the other hand
\[
\dim_{\F_{p}} \Soc^{2}(J) / \Soc(J) = \frac{(n-2)(n^2 +5n +3)}{3}  \, . \] \end{prop}

\begin{proof}Rather than working with the socle series, we will study the radical series, indeed remember this small fact (\cite[Exercise 6.7]{webb})
\[(\Soc^{k}(M)/\Soc^{k-1}(M))^* \simeq\Rad^{k-1}(M^*)/\Rad^{k}(M^*)\, . \] Since we seek to compute the dimension, we may as well compute $\Rad^{k-1}(M^*)/\Rad^{k}(M^*)$.
We claim that, in this case a basis of $\Rad (\omega_2 (\F_p))/\Rad^{2}(\omega_2(\F_p)$ is given by the representatives of those elements
\begin{equation}
\left\lbrace\begin{array}{lr}
(x_i,x_j)X_{i}^{p-1} & 1\leq i<j\leq n \\
(x_i,x_j)X_{j}^{p-1} & 1\leq i<j\leq n \\
(x_i ,x_j)X_k & 1\leq i<j\leq n,k\leq j
\end{array}\right.
 \label{Soc}
 \end{equation}
 
Indeed, $\Rad(J)/\Rad^2 (J)$ is generated by elements of the form $x\cdot X_{i}$, where $x\in J/\Rad(J)$. It is clear that such elements are in fact

\[\left\lbrace\begin{array}{r}
x_{i}^{p}\cdot X_{j}\\
(x_{i},x_{j})X_{k}
\end{array}\right. \]
As pointed out earlier, this is not a free system, but using the relations in the same exact manner as earlier, we can get rid of the elements which are not present in the system which we claimed to be a base. It is therefore clear that the elements of \ref{Soc} are generators of $\Rad(J)/\Rad^2(J)$; since they appear in the given basis, they are linearly independent and hence form a basis.

Let us then count each of them
\[\begin{array}{rcl}
2\cdot\begin{pmatrix}
n \\ 2
\end{pmatrix}+\sum\limits_{j=0}^{n}(j-1)j&=&n(n-1)+\frac{n(n-1)(n+1)}{3} \\
 &=& \frac{n(n-1)(n+4)}{3}
\end{array}\]
  
When $\mathcal{G}_{\k}(p)$ is a \dem group, we have to suppress few elements among those mentioned, indeed we must suppress the $(x_{n-1},x_n)X_{k}$ for $k\in\{1,\ldots,n\}$, furthermore we must get rid of the elements $(x_{n-1},x_{n})X_{n}^{p-1}$, $(x_{n-1},x_{n})X_{n-1}^{p-1}$, thanks to the \dem basis. Therefore

\[\begin{array}{rcl}
\dim\Soc^{2}(J)/\dim\Soc(J)&=&\frac{n(n-1)(n+4)}{3}-n-2 \\
&=&\frac{(n-2)(n^2 +5n+3)}{3} \end{array} \]

The proof is complete.
 \end{proof}
 
\section{Two examples} \label{P4}

In \ref{Cnsq}, we promised to exhibit two concrete extensions $\mathbf{K}_1 /\k$ and $\K_2 /\k$ such that $Gal(\mathbf{K}_1 /\k)\simeq Gal(\K_2 /\k)$, but such that each one of them verifies one case we distinguished in Theorem \ref{MT1}, which means that $J(\K_1)$ is stably isomorphic to $\Omega^{2}(\F_p)\oplus\Omega^{-2}(\F_p)$, whereas $J(\K_2)$ is isomorphic to $\F_{p}^{2}\oplus P$, where $P$ is a projective module.

In order to make our computations clearer, we set $p=3$, but, for every prime $p$, other examples may be found without any difficulty. Consider $\k=\Q_3 (\xi_3)$: since the extension $\Q_3 (\xi_3)/\Q_3$ is of degree $2$, we have
\[\k^{\times} /\k^{\times 3}=\F_{3}^{4} \, , \]
according to \cite[4.10]{Gui}. Hence the following isomorphisms are true:
\[\mathcal{G}_{\k}(3)\simeq Gal(\k(3)/\k)\simeq \mathcal{D}_{1,4}\simeq\langle x_1,x_2,x_3,x_4 | x_{1}^{3}(x_1,x_2)(x_3,x_4) \rangle \, . \]

Now, set
\[\left\lbrace\begin{array}{rcl}
\mathcal{H}_1 &=&\operatorname{Gr}(\Phi(\mathcal{D}_{1,4}),x_1,x_2)\\
\mathcal{H}_2 &=&\operatorname{Gr}(\Phi(\mathcal{D}_{1,4}),x_1,x_4)
\end{array}\right.\, . \]
Let us then write $\K_{i}=\k(3)^{\mathcal{H}_{i}}$ for $i\in\{1,2\}$. Note that $\K_{i}/\k$ is an elementary abelian extension of $\k$, since $\Phi(\mathcal{G}_{\k}(3))$ is a subgroup of $\mathcal{H}_{i}$. Furthermore, the following groups are isomorphic:

\[Gal(\K_{i}/\k)\simeq\mathcal{G}_{\k}(3)/\mathcal{H}_{i}=E_2 \, . \]

In fact, $Gal(\mathbf{K}_{1}/\k)$ is generated by the classes of equivalences of $x_3$ and $x_4$, whereas the group $Gal(\K_2/\k)$ is generated by the classes of $x_2$ and $x_3$.

Remember that, thanks to the presentation we use for a \dem group, there exists a canonical epimorphism
\[\pi\colon\mathcal{F}_4=\langle\tilde{x}_1,\tilde{x}_2,\tilde{x}_3,\tilde{x}_4\rangle\longrightarrow\mathcal{D}_{1,4} \]
sending $\tilde{x}_{i}$ to $x_{i}$. Again, we set $\tilde{\mathcal{H}}_{i}=\pi^{-1}(\mathcal{H}_{i})$ and study rather $\tilde{\mathcal{H}}_{i}/\Phi(\tilde{\mathcal{H}}_{i})$, by giving as usual a presentation.

\emph{Generators.} Remark that $\tilde{\mathcal{H}}_{i}/\Phi(\tilde{\mathcal{H}}_{i})$ is generated, as a module, by a family of (topological) generators of $\tilde{\mathcal{H}}_{i}$. Since $\Phi(\mathcal{F}_{4})$ is generated by the elements $(\tilde{x}_{i},\tilde{x}_{j})$ ( where $1\leq i<j\leq 4$) and $\tilde{x}_{i}^{3}$ (with $1\leq i \leq 4$), we have at our disposal a family of such generators. 

Now in the case of $\tilde{\mathcal{H}}_{1}$ we have a couple more generators, namely $\tilde{x}_1$ and $\tilde{x}_2$; therefore it is clear that $\tilde{x}_{1}^{3}$ and $\tilde{x}_{2}^{3}$ are redundant. Because $Gal(\K_1,\k)\simeq\langle \tilde{x}_3,\tilde{x}_4 \rangle$ acts on $\tilde{\mathcal{H}}_1 /\Phi(\tilde{\mathcal{H}}_1)$, we have for instance
\[(\tilde{x}_{1},\tilde{x}_{3})=\tilde{x}_{1}^{-1}\tilde{x}_{3}^{-1}\tilde{x}_{1}\tilde{x}_{3}=\tilde{x}_{1}^{-1}\tilde{x}_{1}^{\tilde{x}_{3}}\, . \] That is why, we can get rid of $(\tilde{x}_1,\tilde{x}_3)$, and in a similar fashion any commutator implying $\tilde{x}_1$ or $\tilde{x}_2$. Note in fact that in $\tilde{\mathcal{H}}_1/\Phi(\tilde{\mathcal{H}}_1)$ the commutator $(\tilde{x}_1,\tilde{x}_2)$ is simply zero. Thus, we can state the following lemma:

\begin{lm}The $E_2$-module $\tilde{\mathcal{H}}_{1}/\Phi(\tilde{\mathcal{H}}_{1})$ (resp. $\tilde{\mathcal{H}}_{2}/\Phi(\tilde{\mathcal{H}}_{2})$) is generated by the following elements: $\tilde{x}_{1},\tilde{x}_{2},\tilde{x}^{3}_{3},\tilde{x}_{4}^{3},(\tilde{x}_{3},\tilde{x}_{4})$ (resp. $\tilde{x}_1, \tilde{x}_4, \tilde{x}_{2}^{3},\tilde{x}_{3}^{3},(\tilde{x}_2,\tilde{x}_3)$).  \end{lm}

\emph{Relations} It is quite clear that we are allowed to re-write the relations proven in the previous section. Therefore it appears easily that we have the following isomorphism

\[\tilde{\mathcal{H}}_1 /\Phi(\tilde{\mathcal{H}}_1)=\Omega^{2}(\F_3)\oplus\langle \tilde{x}_1 \rangle\oplus\langle \tilde{x}_2\rangle \, , \]

where $\langle \tilde{x}_{i}\rangle$ denotes a projective summand generated by $\tilde{x}_{i}$. Using the previous formulae, $\Omega^2 (\F_3)$ is the module whose generator are simply $\tilde{x}_{3}^{3},\tilde{x}_{4}^{3},(\tilde{x}_3,\tilde{x}_4)$ and the relations are the now-well-known:

\[\left\lbrace\begin{array}{rcll}
\tilde{x}_{i}^{p}X_{i}&=&0&(i\in\{3,4 \}) \\
\tilde{x}_{3}^{3}X_{4}&=&(\tilde{x}_3,\tilde{x}_4)X_{3}^2 & \\
\tilde{x}_{4}^{3}X_{3}&=&-(\tilde{x}_3,\tilde{x}_4)X_{4}^2 &

\end{array}\right. \]

Using the same exact arguments we can get a similar description of $\tilde{\mathcal{H}}_2/\Phi(\tilde{\mathcal{H}}_2)$. 

Let us now describe the map $\kappa$ in each case: in order to distinguish properly the different cases, we set $\kappa_{i}\colon \tilde{\mathcal{H}}_{i}/\Phi(\tilde{\mathcal{H}}_{i})\longrightarrow \mathcal{H}_{i}/\Phi(\mathcal{H}_{i})$.

\begin{prop}The map
\[\kappa_{i}\colon\Omega^{-1}(\F_3)\longrightarrow\tilde{\mathcal{H}}_{i}/\Phi(\tilde{\mathcal{H}}_{i}) \]
is stably zero for $i=2$ and non-zero for $i=1$.

\end{prop}

\begin{proof}

Remember that the source of the morphism $\kappa_{i}$ is simply the module
\[\Omega^{-1}(\F_3)\simeq\langle\alpha|\alpha\cdot\No=0 \rangle\, , \]and its target is in fact $\tilde{\mathcal{H}}_{i}/\Phi(\tilde{\mathcal{H}}_{i})$. To give full details, it is defined by $\kappa_{i}(\alpha)=[\Delta]_{i}$ where $[\Delta]_{i}$ is the equivalence class of $\Delta= \tilde{x}_{1}^{3}(\tilde{x}_1,\tilde{x}_2)(\tilde{x}_3,\tilde{x}_4)$ modulo $\Phi(\tilde{\mathcal{H}}_{i})$. However, according to the previous presentations, in one hand we have 
\[[\Delta]_{1}=(\tilde{x}_3,\tilde{x}_4)\, , \]
whereas in the other hand
\[[\Delta]_{2}=\tilde{x}_{1}X_2+\tilde{x}_{4}X_3\, .\]

Now, note that $\kappa_2\colon\Omega^{-1}(\F_3)\longrightarrow\tilde{\mathcal{H}}_{2}/\Phi(\tilde{\mathcal{H}}_{2})$ factors through a projective module since its image is a subset of $\langle \tilde{x}_1 \rangle\oplus \langle \tilde{x}_4 \rangle$ which is a free module. Therefore
\[J(\K_2)=\Omega^2 (\F_3)\oplus\Omega^{-2}(\F_3)\, . \]

According to the previous computation the module $J(\K_1)$ is in fact isomorphic to
\[(\F_3\times\F_3)\oplus\langle x_1\rangle\oplus\langle x_2 \rangle\, . \]
The two copies of the trivial modules are generated by $x_{3}^3$ and $x_{4}^3$: indeed, in $J(\K_1)=\mathcal{H}_{1}/\Phi(\mathcal{H}_1)$, we have that $x_{3}^{3}\cdot X_4= 0=x_{3}^{3}\cdot X_{3}$, since $(x_3,x_4)=0$ thanks to the \dem relation. In a similar manner $x_{4}^{3}X_{4}=x_{4}^{3}X_{3}=0$, so that a presentation by generators and relation of $J(\K_1 )$ is just:
\[\begin{array}{rcl}
J(\K_1)&=&\langle x_1 , x_2, x_{3}^{3},x_{4}^{3}|x_{i}^{3}X_{j}=0\, ,\; \; i,j\in\{ 3,4\}\rangle \\
&\simeq&\F_3\times\F_3\oplus\langle x_{1} \rangle \oplus\langle x_{2} \rangle
\end{array}
\]
Thus $J(\K_1)$ is stably isomorphic to $\F_3\times\F_3$ which is not stably isomorphic to $\Omega^{2}(\F_3)$, hence $\kappa_1$ is stably non-zero.

\end{proof}

\begin{rem} A very peculiar phenomenon occurs here: since $G=E_2$, if $\kappa$ is stably non-zero, then the module $J(\K)$ is but $\F_p \times \F_p$! As previously seen, the structure of $J(\K)$ is in general far from trivial. The careful reader may have an impression of déjà-vu, indeed this remark has to be linked to the one ending the subsection \ref{Remarque}.\end{rem}

\bibliographystyle{alpha}
\bibliography{biblio}

\end{document}